\documentclass[11pt,a4paper,twoside]{article}

\usepackage[utf8]{inputenc}
\usepackage[margin=2cm]{geometry}
\usepackage[shortlabels]{enumitem}
\usepackage[page,toc]{appendix}
\usepackage{amsmath,amssymb,graphicx,xcolor,subcaption,float,amsthm,mathrsfs,mathtools,bbold,framed,url,amscd,soul}
\usepackage[color=yellow]{todonotes}

\setstcolor{red}

\usepackage{hyperref}
\hypersetup{
colorlinks=true,
}

\newcommand{\nocontentsline}[3]{}
\newcommand{\tocless}[2]{\bgroup\let\addcontentsline=\nocontentsline#1{#2}\egroup}

\theoremstyle{plain}
\newtheorem{theorem}{Theorem}[section]

\newtheorem{lemma}{Lemma}[section]
\newtheorem{proposition}{Proposition}[section]

    \newtheoremstyle{TheoremNum}
        {\topsep}{\topsep}              
        {\itshape}                      
        {}                              
        {\bfseries}                     
        {.}                             
        { }                             
        {\thmname{#1}\thmnote{ \bfseries #3}}
    \theoremstyle{TheoremNum}

    \newtheorem{repeat_lemma}{Lemma}
    \newtheorem{repeat_proposition}{Proposition}

\theoremstyle{definition}
\newtheorem{definition}{Definition}[section]
\theoremstyle{remark}
\newtheorem{remark}{Remark}[section]

\usepackage{fancyhdr}
\emergencystretch=\maxdimen
\begin{document}
	\title{On the analytical aspects of inertial particle motion}
	\date{\today}
	\author{Oliver D. Street \thanks{Department of Mathematics, Imperial College London.} \and Dan Crisan \footnotemark[1]}
	\maketitle
	\begin{abstract}
	    In their seminal 1983 paper, M. Maxey and J. Riley introduced an equation for the motion of a sphere through a fluid. Since this equation features the Basset history integral, the popularity of this equation has broadened the use of a certain form of fractional differential equation to study inertial particle motion. 
	    In this paper, we give a comprehensive theoretical analysis of the Maxey-Riley equation. In particular, we build on previous local in time existence and uniqueness results to prove that solutions of the Maxey-Riley equation are global in time. In doing so, we also prove that the notion of a maximal solution extends to this equation. We furthermore prove conditions under which solutions are differentiable at the initial time. By considering the derivative of the solution with respect to the initial conditions, we perform a sensitivity analysis and demonstrate that two inertial trajectories can not meet, as well as provide a control on the growth of the distance between a pair of inertial particles. The properties we prove here for the Maxey-Riley equations are also possessed, mutatis mutandis, by a broader class of fractional differential equations of a similar form.
	    
	\end{abstract}

\tableofcontents

\section{Introduction}

Often, when modelling the movement of inertial particles in the ocean, the assumption is made that the object's mass does not influence its trajectory with no thorough justification of whether this is a good assumption or not. A detailed summary of current methods in this field can be found in \cite{VanSebille2018}. Purely data driven approaches without physical modelling can provide new insight into environmental issues connected to this topic, including ocean plastics \cite{Sebille2012}. Given the enormous environmental impact of small plastic pollutants (see e.g. \cite{Sebille2015}), a better mathematical understanding of models related to this phenomenon are crucial.  

In the development of an equation which accounts for the effects of mass, the Maxey-Riley equation \cite{Maxey1983} has dominated literature since its introduction in 1983 \cite{Michaelides1997}. Following on from a number of historical results dating back as far as the 19th century \cite{Basset1888,Boussinesq1885,Faxen1922,Oseen1927,Stokes1851,Tchen1947}, M. Maxey and J. Riley conglomerated these results into a useable equation which has proven very appealing to researchers as an `off-the-shelf' solution. Unfortunately, the Maxey-Riley equation is difficult to implement in practice. Specifically, the equation features temporal memory and thus presents data storage issues when numerically implemented. Furthermore, the equation is nonlinear and features a nonlocal integral term. These difficulties have been avoided in some studies by ignoring the Basset history term \cite{Michaelides1997}, or the so-called Fax\'en corrections \cite{Haller2008}. These simplifications look to be questionable in light of a growing body of evidence in support of the role of the Basset history term \cite{Daitche2015,Daitche2014,Guseva2016}. It's worth noting that making crude simplifications to the Maxey-Riley equation may appear reasonable on the surface, however throwing away any term in the Maxey-Riley equation is equivalent to ignoring the corresponding result from the historic literature. Nonetheless, authors have attempted to use a simplified version of Maxey-Riley for ocean transport applications \cite{Beron-Vera2015}. Attempts have been made to circumvent the need for the Basset history term by using a simplified Maxey-Riley equation and including a stochastic noise to match the equation to experimental data \cite{Sapsis2011}. A recent paper \cite{Prasath2019} has given hope to the idea of numerically solving the `full' Maxey-Riley equation with memory, and attempts have been made to use the Maxey-Riley equation as inspiration to create a framework more tailored to oceanography applications \cite{Beron-Vera2020a,Beron-Vera2019,Olascoaga2020}. As a separate issue, the suitability of the Maxey-Riley equation for particles in the ocean remains contentious due to the assumption on the size of the Reynolds number in the seminal paper \cite{Maxey1983}. 

From an analytical perspective, little is known about the Maxey-Riley equation or other fractional differential equations of this type. In \cite{Farazmand2015} the local existence and uniqueness of weak solutions of the Maxey-Riley equation is shown. Also, if the solution is differentiable at its initial time, the equations of motion can be re-written into a form which does permit strong solutions. In this paper, we present a comprehensive analysis on the Maxey-Riley equation.

We do so by considering the Maxey-Riley model as a fractional differential equation, and use methodology from fractional calculus to address complications caused by the Basset history term. This follows from the observation that the Basset history term takes the form of a fractional derivative of Riemann-Liouville type \cite{Tatom1988}.

In this work we show that many of the classical properties of ordinary differential equations (ODEs) also apply to the Maxey-Riley equation, and in doing so we provide detailed analytical properties for the Maxey-Riley equation. Since the complications in achieving this stem almost entirely from the Basset history term, these properties will extend to a class of equations featuring this term. We will address the issue of global existence and uniqueness of a weak solution, as well as introduce precise conditions under which a unique global strong solution exists. Furthermore, we will perform a sensitivity analysis with respect to the initial condition, thus establishing a restriction on how far `nearby particles' can move apart, as well as prove that two inertial particles with different initial conditions can not collide within the Maxey-Riley model. A better understanding of models featuring memory with this structure will assist in numerical studies by highlighting certain solution properties and thus enable researchers to understand what properties of ODEs are preserved in spite of the effect of memory on the solution. 

\paragraph{Contribution of the paper.}
\begin{itemize}
    \item In section \ref{GLOBALWELLPOSEDNESSSECTION}, we introduce assumptions (weaker than those assumed in \cite{Farazmand2015}) under which the Maxey-Riley equation has local weak solutions. Next, we prove the main result of the section; that solutions of Maxey-Riley are global in time, building on previous local existence and uniqueness results \cite{Farazmand2015}. In doing this, we also prove that the classical notion of a `maximal solution' extends to the Maxey-Riley equation and, in general, to fractional order equations. A Gr\"onwall inequality for fractional differential equations \cite{Lin2013a} (see Appendix \ref{GronwallAppendix}) is used to provide the necessary control on the solutions.
    \item In section \ref{WEAKTOSTRONGSOLUTIONS}, we cover the regularity of the solutions of the Maxey-Riley equation. Much of the difficulty here lies in the behaviour of the fractional order Basset history term at its lower limit $t_0$. Conditions under which solutions are differentiable are introduced.
    \item In section \ref{InitialConditionsSection}, we examine the equation which governs the evolution of the derivative of the solution to the Maxey-Riley equation with respect to the initial conditions. This enables us to perform a sensitivity analysis. A bound on this derivative gives a control on the distance between a pair of inertial particles.
    
    We derive the equation governing the matrix inverse of the derivative with respect to the inital conditions, and we show that this derivative is also bounded. This boundedness proves that two inertial trajectories, with distinct initial conditions, governed by the Maxey-Riley equation can not meet. 
\end{itemize}

\section{Preliminaries and existing analytical results}

\subsection{Framework}

\paragraph{The equations of motion.} The results we prove here extend those in Farazmand and Haller \cite{Farazmand2015}. To ensure clarity we will use much of the same notation as in \cite{Farazmand2015}.

For a fluid moving in a domain $\mathscr{D}\subseteq\mathbb{R}^n$ with velocity field $u:\mathscr{D}\times[0,\infty)\to \mathbb{R}^n$, we denote the trajectory of a inertial particle with mass released at time $t_0$ by $y:[t_0,\infty)\to\mathbb{R}^n$, and its velocity by $v:[t_0,\infty)\to\mathbb{R}^n$.  We nondimensionalise our problem by length scale $L$, time scale $T$, and velocity $U$ which are characteristic to the ambient flow $u$. For this flow, we have a Reynolds number (Re) and, for a particle of radius $a$, the problem corresponds to a Stokes number (St) where these quantities are defined via the kinematic viscosity $\nu$ by
\begin{equation}
    \rm{Re} = \frac{UL}{\nu},\quad \rm{St} = \frac{2}{9}\frac{a^2}{\nu T}.
\end{equation}
In a frame of reference moving with the particle, the Maxey-Riley equation may be written in the following form
\begin{equation}\label{MR}
    \begin{aligned}
    \dot{y} &= v \\
    \dot{v} &= R\frac{Du}{Dt}+\left(1-\frac{3R}{2}\right)g + \frac{R}{2}\frac{D}{Dt}\left(u + \frac{\gamma}{10}\mu^{-1}\Delta u \right) \\
    &\quad - \mu\left(v-u-\frac{\gamma}{6}\mu^{-1}\Delta u\right) - \kappa\mu^{1/2}\frac{d}{dt}\int_{t_0}^t\frac{w(s)}{\sqrt{t-s}}\,ds,
    \end{aligned}
\end{equation}
where
\begin{equation}
    w(t)=\dot{y}(t) - u(y(t),t) - \frac{\gamma}{6}\mu^{-1}\Delta u(y(t),t),
\end{equation}
and the additional parameters are defined by
\begin{equation}
    R=\frac{2\rho_f}{\rho_f+2\rho_p},\quad \mu = \frac{R}{\rm{St}},\quad \kappa = \sqrt{\frac{9R}{2\pi}},\quad \gamma = \frac{9R}{2\rm{Re}}.
\end{equation}
As in \cite{Farazmand2015}, we write equation \eqref{MR} in the following form
\begin{equation}\label{MR_ODEs}
    \begin{aligned}
        \dot{y} &= w + A_u(y,t), \\
        \dot{w} &= -\mu w - M_u(y,t)w-\kappa\mu^{1/2}\frac{d}{dt}\int_{t_0}^t\frac{w(s)}{\sqrt{t-s}}\,ds + B_u(y,t),
    \end{aligned}
\end{equation}
where $A_u,B_u:\mathscr{D}\times[t_0,\infty)\to\mathbb{R}^n$ and $M_u:\mathscr{D}\times[t_0,\infty)\to\mathbb{R}^{n\times n}$ are defined by
\begin{equation}
    \begin{aligned}
        A_u &= u + \frac{\gamma}{6}\mu^{-1}\Delta u, \\
        B_u &= \left(\frac{3R}{2}-1\right)\left(\frac{Du}{Dt}-g\right) + \left(\frac{R}{20}-\frac{1}{6}\right)\gamma\mu^{-1}\frac{D}{Dt}\Delta u \\
        &\quad -\frac{\gamma}{6}\mu^{-1}\left(\nabla u + \frac{\gamma}{6}\mu^{-1}\nabla\Delta u\right)\Delta u, \\
        M_u &= \nabla u + \frac{\gamma}{6}\mu^{-1}\nabla\Delta u.
    \end{aligned}
\end{equation}
\paragraph{Well-posedness properties.} We consider the integrated version of \eqref{MR_ODEs}
\begin{equation}\label{MR_integrated_ODEs}
    \begin{aligned}
        y(t) &= y_0 + \int_{t_0}^t w(s) + A_u(y(s),s)\,ds, \\
        w(t) &= w_0 + \int_{t_0}^t \left( -\mu w(s) - M_u(y(s),s)w(s)-\kappa\mu^{1/2}\frac{w(s)}{\sqrt{t-s}} + B_u(y(s),s)\right)\,ds.
    \end{aligned}    
\end{equation}

\begin{remark}
    Equation \eqref{MR_integrated_ODEs} is not a standard ODE, since the integrand of the equation in $w$ has $t$ as an argument. Standard ODE theory can not be applied and we need to develop all notions and results in the new context.
\end{remark}

\begin{definition}
    A solution of \eqref{MR_ODEs} is called \emph{weak} if it satisfies the integrated formulation \eqref{MR_integrated_ODEs}. A solution of \eqref{MR_ODEs} is called \emph{strong} if it satisfies \eqref{MR_integrated_ODEs} and also it is differentiable in time.
\end{definition}

\begin{remark}\label{FarazmandTheorem} 
    Farazmand and Haller \cite{Farazmand2015} prove that \eqref{MR_ODEs} has a weak solution under the following constraints: that $u(x,t)$ is three times continuously differentiable in both $x$ and $t$ and that all of its partial derivatives are uniformly bounded and Lipschitz continuous up to order three. Under these assumptions, for any initial condition $(y_0,w_0)\in\mathscr{D}\times\mathbb{R}^n$, there exists some $T>t_0$ such that over the time interval $[t_0,T)$ the integral equation \eqref{MR_integrated_ODEs} has a unique solution $(y(t),w(t))$ with $(y(t_0),w(t_0)) = (y_0,w_0)$.
    Within the proof of this result, it is also shown that $y,w$ have continuous paths on the interval $[t_0,t_0+T]$, and are hence bounded.
    
    Moreover, in \cite{Farazmand2015} this result is proven for a closed time interval. We state here the results with the half open interval $[t_0,T)$. If we were to include the endpoint $T$ we would have complications defining the derivative at $T$ due to the inability to define the limit from above. Furthermore, the inclusion of the half open interval will allow us to introduce the notion of a maximal solution.

    Here we will relax these conditions (see \eqref{assumptions} below), before extending this result to be global in time.
\end{remark}

\begin{remark}
    One can not determine whether a solution exists for the system of equations \eqref{MR_ODEs} from any general theorems known to the authors from the literature of both ordinary and fractional order differential equations (see e.g. \cite{Podlubny1999}). This is due to the specific nature of the nonlinearity of the system.
\end{remark}

\subsection{Maxey-Riley equation as a fractional differential equation}\label{FractionalSection}

In the context of the Basset-Boussinesq-Oseen equation, in the 1980s it was observed that the Basset history integral is in fact a Riemann-Liouville type  fractional derivative \cite{Tatom1988}. The same remark has been made in the context of the Maxey-Riley equation by a number of authors since \cite{Daitche2015,Daitche2014,Farazmand2015}. To illustrate this, we recall a definition of fractional derivatives \cite{Podlubny1999}.
\begin{definition}\label{RiemannLiouvilleDefinition}
	For a real number $p\in\mathbb{R}$, define the integer $n\in\mathbb{Z}$ to be such that $n-1\leq p < n$. We may then define the \emph{left Riemann-Liouville fractional derivative of order $p$} by
	\begin{equation}\label{RiemannLiouvilleDefinitionEquation}
		_aD^pf(t) = \frac{1}{\Gamma(n-p)}\frac{d^n}{dt^n}\int_a^t(t-s)^{n-p-1}f(s)\,ds.
	\end{equation}
\end{definition}
By comparing \eqref{RiemannLiouvilleDefinitionEquation} with \eqref{MR}, we can immediately see that the history integral is a Riemann-Liouville fractional derivative of order $1/2$
\begin{equation}
	_{t_0}D^{1/2}w(t) = \frac{1}{\sqrt{\pi}}\frac{d}{dt}\int_{t_0}^t\frac{w(s)}{\sqrt{t-s}}\,ds.
\end{equation}
Following \cite{Farazmand2015}, we write the Maxey-Riley equation as a system of nonlinear fractional differential equations 
\begin{equation}\label{MR_FDEs}
    \begin{aligned}
        _{t_0}D^1y(t) &= w + A_u(y,t), \\
        _{t_0}D^1w(t) &= -\mu w - M_u(y,t)w-\kappa\mu^{1/2}\sqrt{\pi}_{t_0}D^{1/2}w(t) + B_u(y,t).
    \end{aligned}
\end{equation}
Whilst this observation has been made as a remark in previous works, its consequences have not been fully exploited until now.

In the 1960s, Caputo \cite{Caputo1967} developed an approach to fractional differential equations which allows for initial value problems to be formulated to involve only the values of integer derivatives of the variables at $t_0$ (see e.g. \cite{Podlubny1999}). This means that, in general, initial value problems for Caputo-type fractional differential equations feature physically interpretable initial conditions. We give below the definition of a Caputo fractional derivative.
\begin{definition}\label{CaputoDefinition}
	For a non-integer real number $p\in\mathbb{R}\setminus\mathbb{Z}$, define the integer $n\in\mathbb{Z}$ to be such that $n-1<p<n$. We may then define the \emph{Caputo fractional derivative of order p} by
	\begin{equation}\label{CaputoDefinitionEquation}
		_a^CD^p_tf(t) = \frac{1}{\Gamma(n-p)}\int_a^t\frac{f^{(n)}(s)}{(t-s)^{p+1-n}}\,ds.
	\end{equation}
\end{definition}
It can be shown (see e.g. \cite{Podlubny1999}) that this derivative is a true interpolation between standard integer order derivatives, i.e. that
\begin{equation}
	\lim_{p\to n}{_a^CD^p_tf(t)} = f^{(n)}(t).
\end{equation}

The system of equations \eqref{MR_FDEs} have a weak solution under conditions \eqref{assumptions} (stated below). However, should the Maxey-Riley equation featuring the Riemann-Liouville fractional derivative have a strong solution, then it will coincide with the version of the Maxey-Riley equation where the Riemann-Liouville fractional derivative is replaced by a Caputo fractional derivative. Farazmand and Haller \cite{Farazmand2015}, show that the latter has a strong solution, nevertheless they do not show the existence of a strong solution of the Maxey-Riley equations featuring the Riemann-Liouville fractional derivative. In this paper, we close this gap and give a clean criterion for the existence of a strong solution. In particular, in Theorem \ref{thm: MR_strong_solutions} we show that the Maxey-Riley equation has a strong solution if and only if $w(t_0)=0$.

\section{Global existence and uniqueness of a weak solution}\label{GLOBALWELLPOSEDNESSSECTION}

    In this section, we extend the local in time existence and uniqueness result in \cite{Farazmand2015} to a global in time result. We do so by using variations of the standard arguments from the theory of ordinary differential equations \cite{Grigorian2007}, as well as a tailored Gr\"onwall lemma for fractional differential equations (see Appendix \ref{GronwallAppendix}). In the following we work with the assumptions:

    \begin{equation}
        \tag{\(*\)}\label{assumptions}
        \parbox{\dimexpr\linewidth-4em}{%
        \strut
        The velocity field, $u$, and its derivatives are sufficiently smooth to ensure that the first derivatives in time and space of $A_u$ and $B_u$ are continuous and uniformly bounded in time and space.%
        \strut}
    \end{equation}
    By uniformly bounded, we mean that its supremum norm is bounded by some constant $L_b$. Thus there exists some constant $L_b$ such that 
    \begin{equation*}
        \|\partial_tA_u\|_\infty,\|\nabla A_u\|_\infty,\|\partial_tB_u\|_\infty,\|\nabla B_u\|_\infty < L_b\,.
    \end{equation*}    
    Notice that the assumption \eqref{assumptions} is sufficient to deduce that $A_u$ and $B_u$ are Lipschitz in space uniformly in time, meaning that there exists some $L_c>0$ such that for any $t\in[t_0,t_0+T)$ and $y_1,y_2\in\mathscr{D}$, we have
  \begin{align*}
      |A_u(t,y_1) - A_u(t,y_2)| &\leq L_c |y_1-y_2| \,,\\
      |B_u(t,y_1) - B_u(t,y_2)| &\leq L_c |y_1-y_2| \,.
  \end{align*}
  Similarly, $A_u$ and $B_u$ are Lipschitz in space. The assumption \eqref{assumptions} is sufficient since the proofs of the lemmata in \cite{Farazmand2015} may be modified to prevent the use of the boundedness of $A_u$ and $B_u$ in the same way that those in Appendix \ref{AppendixFarazmandLemmas} have been modified to prevent this.

\begin{definition}
    A solution $(\tilde{y},\tilde{w})$ of \eqref{MR_integrated_ODEs} with domain $[t_0,\tilde T)$ is called an \emph{extension} of the solution $(y,w)$ with domain $[t_0,T)$ if $t_0<T < \tilde{T}$ and the solutions are identical on $[t_0,T)$. The solution $(y,w)$ is called \emph{maximal} if there exists no such extension.
\end{definition}

\begin{proposition}\label{UniquenessProposition}
    Suppose $(y_1,w_1)$ and $(y_2,w_2)$ are two solutions to \eqref{MR_integrated_ODEs} with domains $[t_0,T_1)$ and $[t_0,T_2)$ respectively, corresponding to the same initial condition $(y_0,w_0)$, then the two solutions coincide on $[t_0,\min\{T_1,T_2\})$.
\end{proposition}

\begin{proof}
See Appendix \ref{TechnicalProofsAppendix}.
\end{proof}

\begin{lemma}\label{UnionSolutionsLemma}
    Let $\{(y_\alpha(t),w_\alpha(t))\}_{\alpha\in A}$ be a family of solutions to \eqref{MR_integrated_ODEs} with initial condition $(y_0,w_0)$, where $A$ is an arbitrary index set. Let the domain of $(y_\alpha,w_\alpha)$ be $[t_0,T_\alpha)$. We define $T$ such that $[t_0,T) = \bigcup_{\alpha\in A}[t_0,T_\alpha)$, and then define a function on $[t_0,T)$ by
    \begin{equation}\label{LemmaSolution}
        (y(t),w(t)) =(y_\alpha(t),w_\alpha(t)),\quad \text{if }t\in [t_0,T_\alpha).
    \end{equation}
    Then $(y(t),w(t))$ is also a solution to \eqref{MR_integrated_ODEs} with the same initial condition.
\end{lemma}

\begin{proof}
    
    See Appendix \ref{TechnicalProofsAppendix}.
\end{proof}

\begin{proposition}\label{MaximalSolutionProposition}
    Assume that $u$ satisfies the conditions in \eqref{assumptions}, then we have a unique maximal solution to \eqref{MR_integrated_ODEs} with initial condition $(y_0,w_0)$.
\end{proposition}
\begin{proof}
    We need only prove that the solution identified in Lemma \ref{UnionSolutionsLemma} corresponding to the family of all possible solutions to \eqref{MR_integrated_ODEs} with initial condition $(y_0,w_0)$ is the unique maximal solution. We know that $(y,w)$ is indeed a solution of \eqref{MR_integrated_ODEs} and it is maximal since its domain contains the domains of all other possible solutions. It only remains to prove uniqueness.
    
    Let $(\tilde y, \tilde w)$ be another such maximal solution. Similar to Lemma \ref{UnionSolutionsLemma}, the union of our two maximal solutions is a solution of \eqref{MR_integrated_ODEs} with initial condition $(y_0,w_0)$ and extends $(y,w)$ and $(\tilde y, \tilde w)$. By the definition of maximality, this union must be identical to both $(y,w)$ and $(\tilde y,\tilde w)$ and hence we have uniqueness.
\end{proof}

\begin{theorem}\label{CompactSetTheorem}
Assume that $u$ satisfies the conditions \eqref{assumptions}, then if $(y(t),w(t))$ is a maximal solution with domain $[t_0,T)$ and $T$ is finite, then $(y,w)$ leaves any compact set $S\subset \mathscr{D}\times\mathbb{R}^{n}$ as $t$ approaches $T$.
\end{theorem}
\begin{proof}
    Let $(y,w)$ be a maximal solution of $(y,w)$ to \eqref{MR_integrated_ODEs} with domain $[t_0,T)$ corresponding to an initial condition $(y_0,w_0)$. Assume further that there exists a compact set $S\subset \mathscr{D}\times\mathbb{R}^n$ such that the solution remains inside $S$, i.e. $\forall \tau\in(t_0,T)$, $\exists t_1 \in (\tau,T)$ s.t. $(y(t_1),w(t_1))\in S$. We will find a contradiction and hence conclude that no such $S$ exists.
    
    Take a sequence $\{(y_n,w_n)\}_{n\in\mathbb{N}}$ defined by $(y_n,w_n)\coloneqq (y(t_n),w(t_n))$ for a sequence $t_n\to T$. Furthermore, we assume that $(y_n,w_n)\in S$ for all $n$. Since $S$ is compact, there exists a converging subsequence $\{(y_{n_k},w_{n_k})\}_{k\in\mathbb{N}}$ where $t_{n_k}\to T$. We call the limit of this sequence $(y_T,w_T)$:
    \begin{equation*}
        (y_{n_k},w_{n_k}) \xrightarrow[k\to\infty]{} (y_T,w_T) \in S
    \end{equation*}
    We may take an element of the sequence which is `arbitrarily close' to $(y_T,w_T)$ in the following way: $\forall \varepsilon>0$ $\exists (y_{t_1},w_{t_1})\in S$ s.t. $|T-t_1|<\varepsilon$. We will pick $t_1$ close to $T$, and use this as an initial condition $(y_{t_1},w_{t_1})$ for a Maxey-Riley equation with memory starting at a time before $t_1$, at $t_0$. In the setup, we have that $(y,w)$ is given on $[t_0,t_1)$ (and indeed beyond this to T) and by construction our Maxey-Riley equation starting at $t_1$ will be shown to extend our solution beyond $T$, hence contradicting maximality.
    
    We have that $w$ at $t_1$ is given by
    \begin{equation}
        w(t_1) = w_0 + \int_{t_0}^{t_1}-\mu w(s) - M_u(y(s),s)w(s) + B_u(y(s),s)\,ds - \kappa \mu^{1/2}\int_{t_0}^{t_1}\frac{w(s)}{\sqrt{t_1-s}}\,ds.
    \end{equation}
    If $w$ is extendable beyond $T$, then for $t>T$ we would have
    \begin{equation}
        w(t) = w_0 + \int_{t_0}^{t}-\mu w(s) - M_u(y(s),s)w(s) + B_u(y(s),s)\,ds - \kappa \mu^{1/2}\int_{t_0}^{t}\frac{w(s)}{\sqrt{t-s}}\,ds,
    \end{equation}
    and thus
    \begin{equation}\label{MR_starting_at_t1}
    \begin{split}
            w(t) - w(t_1) &= \int_{t_1}^{t}-\mu w(s) - M_u(y(s),s)w(s) + B_u(y(s),s)\,ds \\
            &\qquad- \kappa  \mu^{1/2}\int_{t_0}^{t}\frac{w(s)}{\sqrt{t-s}}\,ds + \kappa \mu^{1/2}\int_{t_0}^{t_1}\frac{w(s)}{\sqrt{t_1-s}}\,ds \\
            &= \int_{t_1}^{t}-\mu w(s) - M_u(y(s),s)w(s) + B_u(y(s),s) -\kappa\mu^{1/2}\frac{w(s)}{\sqrt{t-s}}\,ds \\
            &\qquad + \kappa\mu^{1/2}\int_{t_0}^{t_1}\frac{w(s)}{\sqrt{t_1-s}} - \frac{w(s)}{\sqrt{t-s}}\,ds.
    \end{split}
    \end{equation}
    We will consider equation \eqref{MR_starting_at_t1} together with
    \begin{equation}\label{MR_starting_at_t1_Yequation}
        y(t) = y(t_1) + \int_{t_1}^t w(s) + A_u(y(s),s)\,ds.
    \end{equation}
    We want to prove that this system has solutions on an interval of length $\delta$ depending only the compact set $S$ and not on $t_1$. We define the map
    \begin{align}\label{Pmap}
        (P\Phi)(t) &= \begin{pmatrix}
                       (P\Phi)_1(t) \\
                       (P\Phi)_2(t)
                      \end{pmatrix} 
    \end{align}
    where
    \begin{align}
        (P\Phi)_1(t) &=  y_{t_1} + \int_{t_1}^t \eta(s) + A_u(\xi(s),s)\,ds, \\
        \begin{split}
            (P\Phi)_2(t) &= w_{t_1} + \int_{t_1}^t - \left(\mu + \frac{\kappa\mu^{1/2}}{\sqrt{t-s}} + M_u(\xi(s),s)\right)\eta(s) + B_u(\xi(s),s)\,ds  \\
            &\qquad+ \kappa\mu^{1/2}\int_{t_0}^{t_1} \frac{w(s)}{\sqrt{t_1-s}} - \frac{w(s)}{\sqrt{t-s}}\,ds.
        \end{split}
    \end{align}
    Note that a solution to equations \eqref{MR_starting_at_t1} and \eqref{MR_starting_at_t1_Yequation} corresponds to a fixed point of the map $P$. We define $R$ to be such that $S\subseteq \bar B_0 (R)$, then lemmata \ref{NewFarazmandLemma1} and \ref{NewFarazmandLemma2} from Appendix \ref{AppendixFarazmandLemmas} give that, for $K = 4\max\{R,2R\sqrt{T-t_0}\}$ and any $\delta$ chosen such that
    \begin{align*}
        \delta+\mu\delta + 2\kappa\mu^{1/2}\sqrt{\delta} + L_b\delta < \delta+\mu\delta + 2\kappa\mu^{1/2}\sqrt{\delta} + 3L_b\delta &< 1/5 \,,\\
        (2+K)L_c\delta &< 1/4 \,,\\
        (2L_b\delta + A_u(0,t_0) + B_u(0,t_0))\delta &< K/4\,,
    \end{align*}
    the map $P$ has a unique fixed point.

    To complete our proof, notice that $\delta$ here depends only on the Euclidean norm of the initial conditions, i.e. on the compact set $S$. Hence we may choose $t_1$ to be within a distance $\delta$ from $T$ and thus we have extended our solution beyond the supposedly maximal domain. We have found the required contradiction and proven our theorem.
\end{proof}

\begin{theorem}\label{GlobalWeakSolutionTheorem}
    Assume that $u$ satisfies the conditions \eqref{assumptions}, then for any initial condition $(y_0,w_0)\in\mathscr{D}\times\mathbb{R}^n$, there exists a unique global solution $(y(t),w(t))$ (i.e. a solution on $[t_0,\infty)$) to the integral equation \eqref{MR_integrated_ODEs} with $(y(t_0),w(t_0))=(y_0,w_0)$.
\end{theorem}
    To prove Theorem \ref{GlobalWeakSolutionTheorem}, we must first introduce a Lemma which finds an appropriate bound on the solution to \eqref{MR_integrated_ODEs}.
\begin{lemma}\label{GronwallLemma}
    If $u$ satisfies the conditions \eqref{assumptions} and $(y,w)$ satisfies \eqref{MR_integrated_ODEs}, on the interval $[t_0,T)$, there exists some $(C_Y,C_W)$ depending on $T,y_0,w_0,\kappa,\mu$ and $L_b$ such that \begin{align}
        \sup_{t\in [t_0,T)}|y(t)|&\leq C_Y,\\
        \sup_{t\in [t_0,T)}|w(t)|&\leq C_W.
    \end{align} 
\end{lemma}
\begin{proof}[Proof (of Lemma \ref{GronwallLemma}).]
See Appendix \ref{TechnicalProofsAppendix}.
\end{proof}
\begin{proof}[Proof (of Theorem \ref{GlobalWeakSolutionTheorem}).]
    Let $(y,w)$ be the unique maximal solution from Proposition \ref{MaximalSolutionProposition} and $[t_0,T)$ its domain. We aim to show that $[t_0,T)=[t_0,\infty)$. Assume the contrary is true, then $T$ is finite. By Theorem \ref{CompactSetTheorem}, $(y,w)$ leaves any compact set $S\in\mathscr{D}\times\mathbb{R}^n$ as $t\to T$. Take a specific compact set
    \begin{equation*}
        S = \bar B_0(r_1)\times \bar B_0(r_2).
    \end{equation*}
    For $t$ sufficiently close to $T$, we know $\|y\|>r_1$ and $\|w\|>r_2$. Since $r_1,r_2$ were chosen arbitrarily, we may deduce that
    \begin{equation*}
        \|y\|,\|w\| \to \infty,\quad \text{as } t\to T.
    \end{equation*}
    On the contrary, we have boundedness of $y$ and $w$ from Lemma \ref{GronwallLemma}. Thus we have reached a contradiction and proven our theorem.
\end{proof}

\section{From weak to strong solutions}\label{WEAKTOSTRONGSOLUTIONS}

In the previous section we showed the global existence of a weak solution of equation \eqref{MR_integrated_ODEs}, extending the known local in time result. We will now explore the observation in \cite{Farazmand2015} that strong solutions exist under more restrictive conditions and `mild' solutions (which we will refer to here as weak) exist in general. The interesting relationship between this result and the Maxey-Riley equation in the context of fractional derivatives is discussed in section \ref{FractionalSection}.

In \cite{Farazmand2015} it is stated without proof that solutions to \eqref{MR_integrated_ODEs} are not necessarily (continuously) differentiable and hence are not, in general, also solutions to \eqref{MR_FDEs}. Farazmand and Haller go on to explain that, if continuously differentiable solutions to the differential form of the equation \eqref{MR_FDEs} exist, then under the special initial condition $w(t_0)=0$ the Basset history term takes the form
\begin{equation}\label{HistoryRLtoCaputo}
	\frac{d}{dt}\int_{t_0}^t\frac{w(s)}{\sqrt{t-s}}\,ds = \int_{t_0}^t\frac{\dot w(s)}{\sqrt{t-s}}\,ds.
\end{equation}
Reviewing Definitions \ref{RiemannLiouvilleDefinition} and \ref{CaputoDefinition}, we see that this is equivalent to saying that, in this special case, the Riemann-Liouville fractional derivative takes the form of a Caputo fractional derivative of the same order. Thus, \underline{if} continuously differentiable solutions to \eqref{MR_FDEs} exist \underline{and} $w(t_0)=0$, then \eqref{MR_FDEs} can be written as
\begin{equation}\label{MR_FDEsCaputo}
    \begin{aligned}
        _{t_0}D^1_ty(t) &= w + A_u(y,t), \\
        _{t_0}D^1_tw(t) &= -\mu w - M_u(y,t)w-\kappa\mu^{1/2}\sqrt{\pi}\;_{t_0}^CD^{1/2}_tw(t) + B_u(y,t).
    \end{aligned}
\end{equation}
In \cite{Farazmand2015}, it is proven that this equation \eqref{MR_FDEsCaputo} indeed admits continuously differentiable solutions with $y(t_0)=y_0$ and $w(t_0)=0$. It still remains to prove under which conditions the solution to \eqref{MR_integrated_ODEs} is differentiable at $t_0$. This will involve proving the necessary conditions for the fractional integral
\begin{equation}\label{MRFractionalIntegral}
	\int_{t_0}^t\frac{w(s)}{\sqrt{t-s}}\,ds,
\end{equation}
to be differentiable at $t_0$. Without proving this, the expression \eqref{HistoryRLtoCaputo} does not make sense at $t=t_0$ and thus we cannot say that continuously differentiable solutions to \eqref{MR_FDEsCaputo} are also continuously differentiable solutions to \eqref{MR_FDEs}. In the following, we resolve this issue. 

When considering the differentiability of solutions to \eqref{MR_integrated_ODEs}, we can see that any issues will arise from the history term \eqref{MRFractionalIntegral}. In particular, with the following results we clarify under what assumptions this term is differentiable at $t_0$.

We start with some lemmata. In particular, we study the smoothness properties of the integral \eqref{MRFractionalIntegral} which may assist in proving differentiability.
Recall that since $w$ is a solution to \eqref{MR_integrated_ODEs}, it satisfies the continuity and boundedness properties in Remark \ref{FarazmandTheorem}. We can certainly prove that the integral \eqref{MRFractionalIntegral} is continuous at $t_0$. Indeed, for every $\epsilon > 0$ there exists $\delta > 0$ such that $2K\sqrt{\delta}<\epsilon$. For this $\epsilon,\delta$ we have that for $|t-t_0|<\delta$:
\begin{equation}\label{IntegralContinuityt_0}
	\begin{aligned}
		\left| \int_{t_0}^t\frac{w(s)}{\sqrt{t-s}}\,ds -\int_{t_0}^{t_0}\frac{w(s)}{\sqrt{t_0-s}}\,ds \right| &= \left| \int_{t_0}^t\frac{w(s)}{\sqrt{t-s}}\,ds \right| \\
		&\leq K\int_{t_0}^t\frac{1}{\sqrt{t-s}}\,ds \leq 2K\sqrt{t-t_0} < \epsilon.
	\end{aligned}
\end{equation}
From this, one can deduce that the integral is 1/2-H\"older continuous.

\begin{lemma}\label{HolderLimitLemma}
    Assuming $(y,w)$ is a solution of \eqref{MR_integrated_ODEs}, we have that
    \begin{itemize}
    \item[(i)] the integral \eqref{MRFractionalIntegral} is 1/2-H\"older, i.e. there exists some constant $C>0$ such that for any $t_1,t_2$ with $t_0<t_1<t_2$ we have
\begin{equation}\label{HolderStatementIntegral}
	\left| \int_{t_0}^{t_2}\frac{w(s)}{\sqrt{t_2-s}}\,ds -\int_{t_0}^{t_1}\frac{w(s)}{\sqrt{t_1-s}}\,ds \right| \leq C|t_2-t_1|^{1/2}\,,
\end{equation}

and
    \item[(ii)] if $w(t_0) = 0$ then the following limit exists and is equal to zero
    \begin{equation}
        \lim_{t\to t_0}\frac{w(t)}{\sqrt{t-t_0}} = 0\,.
    \end{equation}
    \end{itemize}
\end{lemma}
\begin{proof}
See Appendix \ref{TechnicalProofsAppendix}.
\end{proof}

\begin{lemma}\label{w_HolderLemma}
    A solution $w(t)$ of \eqref{MR_integrated_ODEs} is $1/2$-H\"older on the interval $[t_0,t_0+\epsilon)$ and locally Lipschitz on $[t_0+\epsilon,\infty)$, for any $\epsilon>0$.
\end{lemma}
\begin{proof}
    By Lemma \ref{HolderLimitLemma} (i), there exists some constant $C>0$ such that for any $t_1,t_2$ with $t_0<t_1<t_2$ we have that the memory term in \eqref{MR_integrated_ODEs} is  $1/2$-H\"older on the interval $[t_0,t_0+\epsilon)$ (see \eqref{HolderStatementIntegral}). As all other terms in \eqref{MR_integrated_ODEs} are Lipschitz continuous, we deduce the  $1/2$-H\"older continuity of $w$. 
    
    It remains to prove that $w(t)$ is Lipschitz on any interval $[t_0+\epsilon,T]$ , for any $\epsilon>0$ and $T>t_0+\epsilon$. Without loss of generality we will assume that $t_0=0$. The definition of $w(t)$ is then
    \begin{equation*}
        w(t) = w(0) + \int_{0}^t \left( -\mu w(s) - M_u(y(s),s)w(s)-\kappa\mu^{1/2}\frac{w(s)}{\sqrt{t-s}} + B_u(y(s),s)\right)\,ds\,, 
    \end{equation*}
    and for $r_0\in [0,1]$ and $t,s\in [\epsilon,\infty)$ we have
    \begin{equation}
    \begin{aligned}
        w(r_0t)-w(r_0s) &= \int_{r_0s}^{r_0t} \left( -\mu w(q) - M_u(y(q),q)w(q) + B_u(y(q),q)\right) \\ &\quad -\kappa\mu^{1/2}\left(\int_0^{r_0t}\frac{w(q)}{\sqrt{r_0t-q}}\,dq - \int_0^{r_0s}\frac{w(q)}{\sqrt{r_0s-q}}\,dq \right)\,.
    \end{aligned}
    \end{equation}
    Making substitutions $q=r_1r_0t$ and $q=r_1r_0s$ in the penultimate and final integrals in the above equation respectively, we have 
    \begin{equation}
    \begin{aligned}
        w(r_0t)-w(r_0s) &= \int_{r_0s}^{r_0t} \left( -\mu w(q) - M_u(y(q),q)w(q) + B_u(y(q),q)\right) \\ &\quad -\kappa\mu^{1/2}\left(\sqrt{r_0t}\int_0^{1}\frac{w(r_1r_0t)}{\sqrt{1-r_1}}\,dr_1 - \sqrt{r_0s}\int_0^{1}\frac{w(r_1r_0s)}{\sqrt{1-r_1}}\,dr_1 \right) \\
        &= \int_{r_0s}^{r_0t} \left( -\mu w(q) - M_u(y(q),q)w(q) + B_u(y(q),q)\right) \\ & -\kappa\mu^{1/2}\left((\sqrt{r_0t}-\sqrt{r_0s})\int_0^{1}\frac{w(r_1r_0t)}{\sqrt{1-r_1}}\,dr_1 + \sqrt{r_0s}\int_0^{1}\frac{w(r_1r_0t)-w(r_1r_0s)}{\sqrt{1-r_1}}\,dr_1 \right)\,,
    \end{aligned}
    \end{equation}
    and from local boundedness of $w$ and the assumptions \eqref{assumptions}, this implies
    \begin{equation}
        |w(r_0t)-w(r_0s)| \leq K_0|t-s| + \kappa\mu^{1/2}\sqrt{r_0s}\int_0^1\frac{|w(r_1r_0t)-w(r_1r_0s|)}{\sqrt{1-r_1}}\,dr_1 \,.
    \end{equation}
 To obtain the above estimate, we used the fact that  (recall that $r_0\in[0,1]$)
    \[
    |\sqrt{r_0t}-\sqrt{r_0s}|\le \sqrt{r_0}\frac {t-s}{\sqrt{t}+\sqrt{s}}\le \frac {1}{2\sqrt{\epsilon}}|t-s|\,.
    \]
    
    Within the integrand of the above inequality, we may iterate the argument by substituting in the definition of $w$ to evaluate $w(r_1r_0t)-w(r_1r_0s)$. We claim that after iterating $k$ times we have
    \begin{equation}\label{InductiveHypothesis}
    \begin{aligned}
        |w(r_0t)-w(r_0s)| \leq K_k|t-s| &+(\kappa\mu^{1/2}\sqrt{r_0s})^{k+1}\int_0^1 \sqrt{\frac{r_1^k}{1-r_1}}\int_0^1\sqrt{\frac{r_2^{k-1}}{1-r_2}}\int_0^1 \cdots \\
        \cdots&\int_0^1 \frac{1}{\sqrt{1-r_{k+1}}}\left|w(r_{k+1}r_k\dots r_0t) - w(r_{k+1}r_k\dots r_0s) \right| dr_{k+1}\dots dr_1 \,,
    \end{aligned}
    \end{equation}
    where 
    $\{K_i,i\in\mathbb{N}\}$, $\kappa$ is an appropriately chosen collection of constants.
    We will prove this by induction, where the above calculation acts as a base case. Suppose our inductive hypothesis \eqref{InductiveHypothesis} is true, then we iterate once more by substituting in the difference
    \begin{equation*}
    \begin{aligned}
        \big|w(r_{k+1}r_k\dots r_0t) &- w(r_{k+1}r_k\dots r_0s) \big| \leq K_k|t-s| \\
        &\quad + \kappa\mu^{1/2} \bigg| \big(\sqrt{r_{k+1}\dots r_0t}-\sqrt{r_{k+1}\dots r_0s}\big)\int_0^1\frac{w(r_{k+2}\dots r_0t)}{\sqrt{1-r_{k+2}}}\,dr_{k+2}\bigg| \\
        &\quad + \kappa\mu^{1/2} \bigg|\sqrt{r_{k+1}\dots r_0s}\int_0^1\frac{w(r_{k+2}\dots r_0t)-w(r_{k+2}\dots r_0s)}{\sqrt{1-r_{k+2}}}\,dr_{k+2}\bigg| \\
        &\leq  K_{k+1}|t-s| + \kappa\mu^{1/2}\sqrt{r_{k+1}\dots r_0s}\int_0^1\frac{|w(r_{k+2}\dots r_0t)-w(r_{k+2}\dots r_0s)|}{\sqrt{1-r_{k+2}}}\,dr_{k+2}\,.
    \end{aligned}
    \end{equation*}
    Using our inductive hypothesis \eqref{InductiveHypothesis}, we get
    \begin{equation}\label{InductiveConclusion}
    \begin{aligned}
        |w(r_0t)-w(r_0s)| \leq K_{k+1}|t-s| &+(\kappa\mu^{1/2}\sqrt{r_0s})^{k+2}\int_0^1 \sqrt{\frac{r_1^{k+1}}{1-r_1}}\int_0^1\sqrt{\frac{r_2^{k}}{1-r_2}}\int_0^1 \cdots \\
        \cdots&\int_0^1 \frac{1}{\sqrt{1-r_{k+2}}}\left|w(r_{k+2}\dots r_0t) - w(r_{k+2}\dots r_0s) \right| dr_{k+2}\dots dr_1 \,,
    \end{aligned}
    \end{equation}
    and hence we have proven our claim by induction.
    
    Noting that $r_i\in [0,1]$ for each $i\in\mathbb{N}$, we take supremum over $\{r_i,i\in\mathbb{N}\}$ to deduce that 
    \begin{equation}
        \sup_{\alpha\in[0,1]} |w(\alpha t)-w(\alpha s)| \leq K_k|t-s| + \mathcal{I}_k \sup_{\alpha\in[0,1]} |w(\alpha t)-w(\alpha s)|\,,
    \end{equation}
    where 
    \begin{equation}
        \mathcal{I}_k = (\kappa\mu^{1/2}\sqrt{r_0s})^{k+1}\int_0^1 \sqrt{\frac{r_1^k}{1-r_1}}\int_0^1\sqrt{\frac{r_2^{k-1}}{1-r_2}}\int_0^1 \cdots \int_0^1 \frac{1}{\sqrt{1-r_{k+1}}} dr_{k+1}\dots dr_1\,.
    \end{equation}
    We now observe the integral
    \begin{equation}
        a_k = \int_0^1 \sqrt{\frac{x^k}{1-x}}\,dx = 2\int_0^{\pi/2}\sin^{k+1}\theta\,d\theta\,.
    \end{equation}
    Note that the sequence $(a_k)_k$ converges to zero as $k$ tends to infinity, and does so at the same rate as $1/\sqrt{k}$. 
    Indeed, note that
    \begin{equation*}
        a_k = 2\int_0^{\pi/2}\sin^{k-1}\theta(1-\cos^2\theta)\,d\theta = a_{k-2} - \frac{a_k}{k}\,,
    \end{equation*}
    where we have used that
    \begin{equation*}
        0 = \frac{2\sin^k\theta}{k}\cos\theta \bigg|_0^{\pi/2} = \int_0^{\pi/2} 2\sin^{k-1}\theta\cos^2\theta - \frac{2\sin^{k+1}\theta}{k}\,d\theta = \int_0^{\pi/2} 2\sin^{k-1}\theta\cos^2\theta\,d\theta - \frac{a_k}{k}\,.
    \end{equation*}
    Therefore we have that
    \begin{equation*}
        a_k = \frac{k}{k+1}a_{k-2} = \left( 1 - \frac{1}{k+1} \right)a_{k-2}\,.
    \end{equation*}
    Since $1-y \leq \exp(-y)$ for positive $y$, we have, for $k$ odd
    \begin{equation*}
        a_k \leq \frac{\pi}{2}\exp\left( -\sum_{n=2}^{\frac{k+1}{2}}\frac{1}{2n} \right)\,,
    \end{equation*}
    and the case where $k$ is even is similar. Note that we have used the fact that $a_1 = \pi/2$, in the case where $k$ is even we would instead use $a_0=2$. Therefore
    \begin{equation*}
        \sqrt{k}a_k \leq \frac{\pi}{2}\exp\left(\frac12\log k -\frac12\sum_{n=2}^{\frac{k+1}{2}}\frac{1}{n} \right)\,.
    \end{equation*}
In the following we will use the fact that the sequence $(b_k)_{k}$ defined as 
\[
b_k=\log k -\sum_{n=1}^k \frac{1}{n}
\]
converges to the Euler-Mascheroni constant $\gamma$. We have that 
\begin{equation*}
        \sqrt{k}a_k \leq \frac{\pi}{2}\exp\left(\frac12\log k -\frac12\left(\log(\frac{k+1}{2})-1 - b_{\frac{k+1}{2}} \right)\right)\,.
    \end{equation*}
Observe that the limit on the right hand side converges to $\frac{\pi}{2}\exp\big( \frac{\log(2)}{2} + \frac12 + \frac{\gamma}{2} \big)$. In particular, the sequence on the right hand side is bounded, since it converges, and hence the sequence $\sqrt{k}a_k$ is bounded above. The analysis for the even terms is similar. We have found that there exists $M$ such that 
    \begin{equation*}
        \sqrt{k}a_k \leq M \quad\implies\quad a_k \leq \frac{M}{\sqrt{k}}\,,
    \end{equation*}
    thus $\mathcal{I}_k \rightarrow 0$ as
    \begin{equation*}
        \mathcal{I}_k \leq (M\kappa\mu^{1/2}\sqrt{r_0s})^{k+1} \frac{1}{\sqrt{k!}} \rightarrow 0\,\quad\hbox{as}\quad k\rightarrow \infty\,.
    \end{equation*}
    Hence we choose $k\in\mathbb{N}$ sufficiently large to ensure that
    \begin{equation}
        \mathcal{I}_k < 1\,,
    \end{equation}
    and thus \eqref{InductiveHypothesis} implies
    \begin{equation}
        (1-\mathcal{I}_k)\sup_{\alpha\in [0,1]}|w(\alpha t)-w(\alpha s)| \leq K_k|t-s|\,.
    \end{equation}
    Rearranging gives
    \begin{equation}
        |w(t)-w(s)| \leq \sup_{\alpha\in [0,1]}|w(\alpha t)-w(\alpha s)| \leq \frac{K_k}{1-\mathcal{I}_k}|t-s|\,,
    \end{equation}
    and we have thus proven our claim.
\end{proof}

We are now in a position to prove the main result of the section.

\begin{theorem}\label{thm: MR_strong_solutions}
    Under the assumptions \eqref{assumptions}, there exists a strong solution of the Maxey-Riley equation \eqref{MR_ODEs} if and only if $w(t_0) = 0$.
\end{theorem}

\begin{proof}
    To prove this theorem, note that the existence of classical solutions is equivalent to the differentiability of the integral \eqref{MRFractionalIntegral} at $t_0$, since the differentiability of the remaining terms in \eqref{MR_integrated_ODEs} is trivial. 
    
    Firstly, assume that $w(t_0)\neq 0$. Then by adding and subtracting $w(t_0)$ to the numerator of the integrand of \eqref{MRFractionalIntegral}, we have
    \begin{align*}
    	\int_{t_0}^t\frac{w(s)}{\sqrt{t-s}}\,ds &= \int_{t_0}^t\frac{w(s)-w(t_0)}{\sqrt{t-s}}\,ds+\int_{t_0}^t\frac{w(t_0)}{\sqrt{t-s}}\,ds \\
    	&= \int_{t_0}^t\frac{w(s)-w(t_0)}{\sqrt{t-s}}\,ds + 2w(t_0)\sqrt{t-t_0}\,.
    \end{align*}
    One may observe that the second term is not differentiable at $t_0$, indeed
    \begin{equation*}
    	\lim_{t\to t_0}\frac{2w(t_0)\sqrt{t-t_0}}{t-t_0} = \lim_{t\to t_0}\frac{2w(t_0)}{\sqrt{t-t_0}} = \infty\,.
    \end{equation*}
    Since, by Lemma \ref{w_HolderLemma}, $w$ is $1/2$-H\"older as $t$ approaches $t_0$, we know that for suitably small $t-t_0$ there exists some $c_2$ such that
    \begin{equation}
        \left|\int_{t_0}^t\frac{w(s)-w(t_0)}{\sqrt{t-s}}\right|\,ds\leq \int_{t_0}^t\frac{\left|w(s)-w(t_0)\right|}{\sqrt{t-s}} \leq c_2 \int_{t_0}^t\frac{\sqrt{s-t_0}}{\sqrt{t-s}}\,ds\,.
    \end{equation}
    We have calculated the integral on the right hand side in the proof of Lemma \ref{HolderLimitLemma} (ii), and thus
    \begin{equation}
        \frac{1}{t-t_0}\left|\int_{t_0}^t\frac{w(s)-w(t_0)}{\sqrt{t-s}}\,ds \right| \leq \frac{c_2\pi}{2}\,.
    \end{equation}
    We may conclude that if $w(t_0)\neq 0$, then the integral \eqref{MRFractionalIntegral} is not differentiable at $t_0$ and thus solutions of \eqref{MR_integrated_ODEs} are not differentiable at $t_0$ and can not be classical solutions. Hence the contrapositive is true, and if solutions of \eqref{MR_integrated_ODEs} are classical solutions, then $w(t_0) = 0$.
    
    We next prove the reverse implication, by assuming that $w(t_0) = 0$. By Lemma \ref{HolderLimitLemma} (i), for $w(t_0)=0$ the following function is bounded
    \begin{equation*}
    	q(s) = \frac{w(s)}{\sqrt{s-t_0}},
    \end{equation*}
    since
    \begin{equation*}
    	\frac{|w(s)|}{\sqrt{s-t_0}} = \frac{|w(s)-w(t_0)|}{\sqrt{s-t_0}} \leq C
    \end{equation*}
    by the definition of H\"older continuity. In this case we have
    \begin{equation}
    	\left|\frac{1}{t-t_0}\int_{t_0}^t\frac{\sqrt{s-t_0}\,q(s)}{\sqrt{t-s}}\,ds \right|\leq \frac{C}{t-t_0}\left| \int_{t_0}^t\frac{\sqrt{s-t_0}}{\sqrt{t-s}}\,ds \right| = \frac{C}{t-t_0}\left| \frac{\pi}{2}(t-t_0) \right| = \frac{C\pi}{2}.
    \end{equation}
    By Lemma \ref{HolderLimitLemma} (ii) the following limit exists
    \begin{equation}\label{ImportantLimit}
    	\lim_{s\to t_0}q(s) = 0,
    \end{equation}
    and therefore for $t$ suitably close to $t_0$ we have
    \begin{equation*}
    	\frac{1}{t-t_0}\int_{t_0}^t\frac{\sqrt{s-t_0}(0-\epsilon)}{\sqrt{t-s}}\,ds \leq \frac{1}{t-t_0}\int_{t_0}^t\frac{\sqrt{s-t_0}\,q(s)}{\sqrt{t-s}}\,ds \leq \frac{1}{t-t_0}\int_{t_0}^t\frac{\sqrt{s-t_0}(0+\epsilon)}{\sqrt{t-s}}\,ds
    \end{equation*}
    and thus
    \begin{equation*}
    	-\frac{C\pi}{2}\epsilon \leq \frac{1}{t-t_0}\int_{t_0}^t\frac{\sqrt{s-t_0}\,q(s)}{\sqrt{t-s}}\,ds \leq \frac{C\pi}{2}\epsilon\,.
    \end{equation*}
    Hence the integral \eqref{MRFractionalIntegral} is differentiable at $t=t_0$, with value zero.
    
    It remains to prove differentiability away from the initial time. By Lemma \ref{w_HolderLemma}, $w(t)$ is Lipschitz and thus absolutely continuous, hence there exists a locally bounded measurable function $a:[t_0,\infty)\mapsto \mathbb{R}$ such that
    \begin{equation*}
        w(s) = w(s) - w(t_0) = \int_{t_0}^s a(r)\,dr\,.
    \end{equation*}
    By integrating by parts, we have the identity
    \begin{equation}
    \begin{aligned}
        0 = \bigg( \sqrt{t-s}\int_{t_0}^s a(r)\,dr \bigg)\bigg|_{t_0}^t &= \int_{t_0}^t \frac{d}{ds}\bigg( \sqrt{t-s}\big( w(t)-w(t_0) \big) \bigg)\,ds \\
        &= \int_{t_0}^t \frac{w(s)-w(t_0)}{2\sqrt{t-s}}\,ds + \int_{t_0}^t a(s)\sqrt{t-s}\,ds \,,
    \end{aligned}
    \end{equation}
    and we thus it suffices to prove the differentiability of
    \begin{equation}
        f(t) = \int_{t_0}^t a(s)\sqrt{t-s}\,ds =  -\int_{t_0}^t \frac{w(s)-w(t_0)}{2\sqrt{t-s}}\,ds = -\int_{t_0}^t \frac{w(s)}{2\sqrt{t-s}}\,ds\,.
    \end{equation}
    Notice that if we prove that $f$ is differentiable for any 
    $t> t_0$, then the Basset history term is differentiable and hence so is $w(t)$. The proof is from first principles: We have that
    \begin{equation}
    \begin{aligned}
        f'(t) &= \lim_{\epsilon\rightarrow 0} \frac{f(t+\epsilon)-f(t)}{\epsilon} = \lim_{\epsilon\rightarrow 0}\left[ \frac{1}{\epsilon}\int_{t_0}^{t+\epsilon} a(s)\sqrt{t+\epsilon-s}\,ds - \frac{1}{\epsilon}\int_{t_0}^t a(s)\sqrt{t-s}\,ds \right] \\
        &= \lim_{\epsilon\rightarrow 0} \left[ \frac{1}{\epsilon}\int_t^{t+\epsilon}a(s)\sqrt{t+\epsilon -s}\,ds \right] + \lim_{\epsilon\rightarrow 0}\int_{t_0}^t a(s)\left[ \frac{\sqrt{t+\epsilon - s} - \sqrt{t-s}}{\epsilon} \right]\,ds\,.
    \end{aligned}\label{atlast}
    \end{equation}
    It remains to prove that the limit can be exchanged with the integral. For $\epsilon>0$, note that
    \begin{equation*}
         \bigg| \frac{1}{\epsilon}\int_t^{t+\epsilon}a(s)\sqrt{t+\epsilon -s}\,ds \bigg| \leq \bigg| \frac{1}{\epsilon}\int_t^{t+\epsilon}a(s)\sqrt{\epsilon}\,ds \bigg| \leq (
         \sup_{s\in [t,t+\epsilon]}|a(s)|)\sqrt{\epsilon}\,,
    \end{equation*}
    the boundedness of $a(s)$ gives 
    \begin{equation*}
        \lim_{\epsilon\rightarrow 0} \left[ \frac{1}{\epsilon}\int_t^{t+\epsilon}a(s)\sqrt{t+\epsilon -s}\,ds \right] = 0\,,
    \end{equation*}
    In the last term in \eqref{atlast} we can switch between the integration and the limit with respect to $\epsilon$ by observing that, for any  $\epsilon >0$, we have 
    \[
     \left |a(s) \frac{\sqrt{t+\epsilon - s} - \sqrt{t-s}}{\epsilon}
     \right |
    =  \left| a(s) \frac{t+\epsilon - s - (t-s)}{\epsilon(\sqrt{t+\epsilon - s} +\sqrt{t-s})}\right|
    \le \frac{\sup_{s\in [t_0,t]} |a(s)| }{2\sqrt{t-s}}
    \]
    and the above upper bound is integrable on the interval $[t_0,t]$. 
    
    For $\epsilon<0$ sufficiently small so that $t+\epsilon>t_0$ one shows in a similar manner that 

    \begin{equation*}
        \lim_{\epsilon\rightarrow 0} \left[ \frac{1}{\epsilon}\int_{t+\epsilon}^ta(s)\sqrt{t-s}\,ds \right] = 0\,,
    \end{equation*}
    Also 
    \[
    \int_{t_0}^{t+\epsilon} a(s)\left[ \frac{\sqrt{t+\epsilon - s} - \sqrt{t-s}}{\epsilon} \right]\,ds=
    \int_{t_0}^{t} a(s)q(s,\epsilon)ds
    \]
    where $q(s,\epsilon)=0$ for $s\in [t+\epsilon, t]$ and 
    \[
    0\le q(s,\epsilon)=  \frac{\sqrt{t+\epsilon - s} - \sqrt{t-s}}{\epsilon}
    =   \frac{t+\epsilon - s - (t-s)}{\epsilon(\sqrt{t+\epsilon - s} +\sqrt{t-s})}
    \le \frac{1 }{\sqrt{t-s}}
    \]
and we have, for any $s\in [t_0,t]$,
    \[
     \left |a(s) \frac{\sqrt{t+\epsilon - s} - \sqrt{t-s}}{\epsilon}
     \right |
\le \frac{\sup_{s\in [t_0,t]} |a(s)| }{\sqrt{t-s}}
    \]
and the above upper bound is integrable on the interval $[t_0,t]$. 
    
Hence we have the required differentiability by the dominated convergence theorem, that is, we have explicitly that
    \begin{equation*}
        f'(t) = \int_{t_0}^t \frac{a(s)}{2\sqrt{t-s}}\,ds\,.
    \end{equation*}
\end{proof}


\section{Properties of the solution as a function of the initial conditions}\label{InitialConditionsSection}

In this section we will work under the following assumptions. 
    \begin{equation}
        \tag{\(**\)}\label{assumptions2}
        \parbox{\dimexpr\linewidth-4em}{%
        \strut
        The velocity field, $u$, is four times continuously differentiable and its partial derivatives are Lipschitz continuous up to order four.%
        \strut}
    \end{equation}

\subsection{Behaviour of neighbouring inertial particles}

\begin{proposition}
Under the conditions \eqref{assumptions2}, the distance between two trajectories at any time $t$ is controlled by the difference between their initial conditions. Indeed, for two inital conditions $x_1,x_2 \in \mathbb{R}^{2n}$, there exists some constant $M$ such that
\begin{equation*}
    \begin{aligned}
        \left\| y(t,x_2) - y(t,x_1) \right\| &\leq M\|x_2-x_1\|\,, \\
        \left\| w(t,x_2) - w(t,x_1) \right\| &\leq M\|x_2-x_1\|\,,
    \end{aligned}
\end{equation*}
where the notation $y(t,x_i)$ and $w(t,x_i)$ is used to reflect the dependence of the solution on its initial conditions.
\end{proposition}

\begin{remark}
    This result implies that one can ensure that two trajectories are arbitrarily close at time $t$ by selecting suitably close initial conditions for them.
\end{remark}

\begin{proof}
Suppose $(y(t),w(t))$ denotes a solution of \eqref{MR_ODEs} corresponding to an initial condition $(y_0,w_0)\in\mathbb{R}^{2n}$. We denote the derivatives of $y$ and $w$ with respect to $(y_0,w_0)$ by $Dy$ and $Dw$ respectively. Note that these derivatives are matrix valued and may be considered as a map $[t_0,\infty)\to \mathbb{R}^{n\times2n}$. As in \cite{Farazmand2015}, these derivatives satisfy the equation
\begin{equation}\label{MR_derivatives}
    \begin{aligned}
        Dy(t) &= (I_n|O_n) + \int_{t_0}^t \Big( Dw(s) + \nabla A_u(y(s),s)Dy(s) \Big) \,ds, \\
        Dw(t) &= (O_n|I_n) + \int_{t_0}^t \bigg( -\mu Dw(s) - \mathcal{L}(y(s),w(s),s)Dy(s) - M_u(y(s),s)Dw(s) \\\
        &\qquad\qquad\qquad\qquad\qquad -\kappa\mu^{1/2}\frac{Dw(s)}{\sqrt{t-s}} + \nabla B_u(y(s),s)Dy(s) \bigg)\,ds,
    \end{aligned}
\end{equation}
where $\mathcal{L}$ is an $n$-dimensional square matrix with components defined by
\begin{equation*}
    \mathcal{L}_{ij}(y(s),w(s),s) = \sum_k \frac{\partial M_{ik}}{\partial y_j}(y(s),s)w_k(s)\,,
\end{equation*}
and $I_n$ and $O_n$ denote the $n$-dimensional identity and null matrices respectively.


The equation \eqref{MR_derivatives} has solutions under assumptions \eqref{assumptions2}.
With a similar methodology to that developed in section \ref{GLOBALWELLPOSEDNESSSECTION}, this result may be shown to hold globally in time. Under assumptions \eqref{assumptions2}, we have that $\nabla A_u$, $\mathcal{L}$, $M_u$, and $B_u$ are all bounded and we thus have sufficient conditions to apply a fractional Gr\"onwall argument as in Appendix \ref{GronwallAppendix}. Let
\begin{equation}
    f(t) = \sup_{s\in[t_0,t]}\sum_{i=1}^n |Dy^i(s)| + |Dw^i(s)|\,,
\end{equation}
then there exist constants $C_1,C_2$ such that
\begin{equation}
    f(t) \leq 2n + \int_{t_0}^t C_1f(s)\,ds + \int_{t_0}^t C_2 (t-s)^{-1/2}f(s)\,ds\,.
\end{equation}
We may once again apply the Gr\"onwall inequality from Appendix \ref{GronwallAppendix}. Hence $Dy$ and $Dw$ are bounded above on intervals $[t_0,T)$ for all $T>t_0$. Suppose $M$ is such that $Dy$ and $Dw$ are bounded above by $M$. Then, for two initial conditions $x_1,x_2 \in \mathbb{R}^{2n}$ we have
\begin{equation*}
    \begin{aligned}
        \left\| y(t,x_2) - y(t,x_1) \right\| &= \int_{x_1}^{x_2} Dy(t,z)\,dz \leq M\|x_2-x_1\|\,, \\
        \left\| w(t,x_2) - w(t,x_1) \right\| &= \int_{x_1}^{x_2} Dy(t,z)\,dz \leq M\|x_2-x_1\|\,.
    \end{aligned}
\end{equation*}
\end{proof}

\subsection{Non-collision of inertial particles}

\begin{proposition}
    Under conditions \eqref{assumptions2}, the distance between two trajectories is always strictly positive if their initial conditions are distinct.
\end{proposition}

\begin{proof}
Define the matrix $D\varphi$ by 
\begin{equation*}
    D\varphi(t) \coloneqq \begin{pmatrix} Dy(t) \\ Dw(t) \end{pmatrix}\,
\end{equation*}
then this matrix evolves according to the equation
\begin{equation}
\begin{aligned}
    D\varphi(t) = I_{2n} &- \int_{t_0}^t \begin{pmatrix} O_n & O_n \\ O_n & -\kappa\mu^{1/2}I_n \end{pmatrix}D\varphi(s)\,\frac{ds}{\sqrt{t-s}} \\
    &+ \int_{t_0}^t \begin{pmatrix} \nabla A_u(y(s),s)I_n & I_n \\ \big( \nabla B_u(y(s),s) - \mathcal{L}(y(s),w(s),s) \big)I_n & -\big(\mu + M_u(y(s),s)\big)I_n\end{pmatrix} D\varphi(s)\,ds\,.
\end{aligned}
\end{equation}
Immediately following Appendix \ref{MatrixInverseAppendix}, the inverse of $D\varphi(t)$ evolves according to
\begin{equation}
\begin{aligned}
    D\varphi^{-1}(s) = &I_{2n} -\int_s^t D\varphi^{-1}(r)\begin{pmatrix} O_n & O_n \\ O_n & -\kappa\mu^{1/2}I_n \end{pmatrix}\frac{dr}{\sqrt{r-s}}  \\
    &- \int_{t_0}^s D\varphi^{-1}(r)\begin{pmatrix} \nabla A_u(y(r),r)I_n & I_n \\ \big( \nabla B_u(y(r),r) - \mathcal{L}(y(r),w(r),r) \big)I_n & -\big(\mu + M_u(y(r),r)\big)I_n\end{pmatrix}\,dr\,.
\end{aligned}
\end{equation}
Since left and right Riemann-Liouville derivatives are equivalent up to time reversal, the same Gronwall argument from Appendix \ref{GronwallAppendix} may be applied. This is valid since $u$ is sufficiently smooth and bounded under the assumptions \eqref{assumptions2} for the coefficients to satisfy the required conditions for the Gronwall theorem to hold. Hence there exists some $\widetilde M$ which is an upper bound for $Dy^{-1}$ and $Dw^{-1}$, and hence
\begin{equation}
    0 < \frac{1}{\widetilde M} \leq Dy, Dw \,.
\end{equation}
\end{proof}

\subsection{Time reversability property}

We have so-far shown that the solution of the Maxey-Riley equation was shown to be injective as a map from the initial conditions to the solution at time $t$. In order to prove that this map is surjective we need to show that, for any point $(\widetilde y,\widetilde w)\in\mathbb{R}^{2n}$, there exists some solution $(y(t),w(t))$ corresponding to an intial condition $(y_0,w_0)$ and some time $T$ such that $(y(T),w(T)) = (\widetilde y, \widetilde w)$. This is related to the idea of time reversal.

Our existence and uniqueness results for the Maxey-Riley equation apply \emph{forwards} in time. Thus, if we know the value of our solution at a particular time $t_0$, our existence and uniqueness results only provide us information about the problem beyond this time, and not before it. To extend on this, we introduce the concept of a \emph{right} fractional derivative defined as follows.
\begin{definition}\label{RightRiemannLiouvilleDefinition}
	For a real number $p\in\mathbb{R}$, define the integer $n\in\mathbb{Z}$ to be such that $n-1\leq p < n$. We may then define the \emph{right Riemann-Liouville fractional derivative of order $p$} by
	\begin{equation}\label{RightRiemannLiouvilleDefinitionEquation}
		D^p_bf(t) = \frac{1}{\Gamma(n-p)}\left(-\frac{d}{dt}\right)^n\int_t^b(s-t)^{n-p-1}f(s)\,ds.
	\end{equation}
\end{definition}
Notice now that this derivative is equivalent to the \emph{left} derivative under time reversal. Indeed, where $s=-\sigma$ and $t=-\tau$, we have
\begin{align*}
    _aD^pf(t) &= \frac{1}{\Gamma(n-p)}\left(\frac{d}{dt}\right)^n\int_a^t(t-s)^{n-p-1}f(s)\,ds \,,\\
    &= \frac{1}{\Gamma(n-p)}\left(-\frac{d}{d\tau}\right)^n\int_{-a}^{\tau}(-\tau+\sigma)^{n-p-1}f(-\sigma)\,(-1)d\sigma \,,\\
    &= \frac{1}{\Gamma(n-p)}\left(-\frac{d}{d\tau}\right)^n\int_{\tau}^{-a}(\sigma-\tau)^{n-p-1}f(-\sigma)\,d\sigma = D_{-a}^p f(-\tau) = D_{-a}^p f(t)\,.
\end{align*}

Thus, by relabeling the time variable to ensure that memory is accumulated starting at time $t=0$ (i.e. $t_0 = 0$), the `backwards' equation for the Maxey-Riley equation has a similar form as the forwards equation, with the Basset history term being a \emph{right} Riemann-Liouville fractional derivative rather than a left one. By existence and uniqueness arguments similar to those for the forward equation, the backwards equation has a unique solution and hence time reversal is possible within this model.

\section*{Acknowledgements}
During this work, O. D. Street has been supported by an EPSRC studentship as part of the Centre for Doctoral Training in the Mathematics of Planet Earth (grant number EP/L016613/1), and D. Crisan has been partially supported by European Research Council (ERC) Synergy grant STUOD - DLV-856408.

We would like to thank our friends and colleagues for generously sharing their thoughts and ideas with us. In particular our thanks go to Darryl D. Holm and Bertrand Chapron, for their insightful comments.

\begin{appendices}
	\section{Gr\"onwall result for fractional differential equations}\label{GronwallAppendix}
	The following version of Gr\"onwall's lemma may be applied to a broad class of fractional differential equations, it may be found as Theorem 1.4 in \cite{Lin2013a}.
	\begin{theorem}\label{LinGronwallLemma}
		If, for any $t\in[0,T)$, we have
		\begin{equation}\label{GronwallAssumption}
		u(t)\leq a(t) + \sum_{i=1}^n b_i(t)\int_0^t(t-s)^{\beta_i-1}u(s)\,ds,
		\end{equation}
		where all the functions are nonnegative and continuous, the constants $\beta_i$ are positive, and $b_i\quad(i=1,2,\dots,n)$ are the bounded and monotonic increasing functions on $[0,T]$. Then, for any $t\in [0,T)$, we have 
		\begin{equation}\label{GronwallConclusion}
		u(t)\leq \sup_{t\in[0,T]}\left(a(t) + \sum_{k=1}^\infty \left(\sum_{1',2',,\dots,k'=1}^n\frac{\prod_{i=1}^{k}[b_{i'}(t)\Gamma(\beta_{i'})]}{\Gamma(\sum_{i=1}^k\beta_{i'})}\int_0^t(t-s)^{\sum_{i=1}^k\beta_{i'} -1}a(s)\,ds \right)\right)<\infty.
		\end{equation}
	\end{theorem}
	
	\begin{remark}
	    For bounded $b_i$ and $a$, the infinite sum in \eqref{GronwallConclusion} converges. To show this, we assume that $a(t)\leq A$ and $b_i(t)\leq B$ for all $i=1,\dots,n$, furthermore we may assume without loss of generality that $\beta_i$ are ordered $\beta_1\leq\beta_2\leq \dots \leq \beta_n$. We label the terms of this series $a_k$, and for $k>2/\beta_1$ we have
	    \begin{align*}
	        a_k &\leq A\sum_{1',2',,\dots,k'=1}^n \frac{B^k(\max_i\Gamma(\beta_i))^k}{\Gamma(\sum_i^k\beta_{i'})}\frac{1}{k\beta_1} \max\{t,1\}^{k\beta_n} \\
	        &\leq Ank \frac{x^k}{\Gamma(k\beta_1)k\beta_1} \qquad\text{where }x\coloneqq B\max\{T,1\}^{\beta_n}\max_i\Gamma(\beta_i) \\
	        &\leq A'\frac{x^k}{\Gamma(k\beta_1)},
	    \end{align*}
	    for a constant $A'$. Note that the inequality in the second line above is only true for $k>2/\beta_1$ since the gamma function is increasing on the interval $[2,\infty)$ however is decreasing nearer to 0. We split $k$ into the following subsets $S_m\coloneqq\{k : m\leq k\beta_1 \leq m+1\}$, and notice that on $S_m$ we have $\Gamma(k\beta_1) > (m-1)!$. Thus we may bound $a_k$ by the terms of the following series
	    \begin{equation*}
	        \sum_{m=2}^\infty \sum_{k\in S_m}A'\frac{x^k}{(m-1)!}.
	    \end{equation*}
	    We have
	    \begin{equation*}
	        \sum_{k\in S_m}x^k \leq \frac{m+1}{\beta_1}y^{\frac{m+1}{\beta_1}},\qquad \text{where }y\coloneqq \max\{x,1\},
	    \end{equation*}
	    and the sum defined by
	    \begin{equation*}
	        \sum_{m=2}^\infty A'\frac{m+1}{\beta_1}\frac{y^{\frac{m+1}{\beta_1}}}{(m-1)!}
	    \end{equation*}
	    obviously converges. Hence the infinite sum in \eqref{GronwallConclusion} converges as claimed.
	\end{remark}

    \section{Lemmata for the map \emph{P}}\label{AppendixFarazmandLemmas}

    Proofs of the following lemmata, which are analogous to lemmata 1 and 2 from \cite{Farazmand2015}, are required to complete the proof of Theorem \ref{CompactSetTheorem}. We define the following space of functions
    \begin{equation*}
        X_{\delta,K} \coloneqq \{f\in C([t_1,t_1+\delta);\mathbb{R}^m):\|f\|\leq K \}
    \end{equation*}
    where $m$ can be either $n$ or $2n$ as required. Note that in the following we will be dealing with the map $P$ defined by \eqref{Pmap}, in the context of which $\Phi$ is $2n$-dimensional and $\eta,\xi$ are $n$-dimensional.
    \begin{lemma}\label{NewFarazmandLemma1}
        For $P$ as defined by equation \eqref{Pmap}, there exists a $K>0$ large enough and $\delta=\delta(K)>0$ small enough and independent of the initial condition such that $P$ maps functions from 
        $X_{\delta,K}$ to $X_{\delta,K}$.
    \end{lemma}
    \begin{proof}
        We must first prove that $P\Phi$ is continuous for continuous $\Phi$, given assumption \eqref{assumptions}.
        This continuity is obvious with the exception of the continuity of the integral
        \begin{equation}\label{ContinuityIntegral1}
            \int_{t_1}^t\frac{\eta(s)}{\sqrt{t-s}}\,ds, 
        \end{equation}
        for $\eta\in X_{\delta,K}$, as well as the integral
        \begin{equation}\label{ContinuityIntegral2}
            \int_{t_0}^{t_1}\frac{w(s)}{\sqrt{t-s}}\,ds,
        \end{equation}
        for $w\in X_{\delta,K}$. Following same argument as in the proof of Lemma \ref{HolderLimitLemma} (i), we may see that \eqref{ContinuityIntegral1} is continuous. It remains only to prove that \eqref{ContinuityIntegral2} has the required continuity. Recalling that $R$ is such that $w\in S\subseteq \bar B_0 (R)$, we have that \eqref{ContinuityIntegral2} is continuous at $\tau$ since for all $\varepsilon > 0$, if we have $|t-\tau|<\frac{\varepsilon^2}{16R^2}$, then
        \begin{align*}
        \left| \int_{t_0}^{t_1}\frac{w(s)}{\sqrt{t-s}}\,ds - \int_{t_0}^{t_1} \frac{w(s)}{\sqrt{\tau-s}}\,ds \right| &= \left| \int_{t_0}^{t_1}\frac{w(s)}{\sqrt{t-s}}-\frac{w(s)}{\sqrt{\tau-s}}\,ds \right| \leq R\left| \int_{t_0}^{t_1}\frac{1}{\sqrt{t-s}}-\frac{1}{\sqrt{\tau-s}}\,ds \right|\\ 
        &= 2R|-\sqrt{t -t_1} + \sqrt{t-t_0} +
        \sqrt{\tau-t_1}-\sqrt{\tau-t_0}| \\
        &\leq 2R|\sqrt{\tau-t_1}-\sqrt{t -t_1}| + 2R|\sqrt{t-t_0} -\sqrt{\tau-t_0}| \\
        &=2R\frac{|\tau -t|}{\sqrt{\tau-t_1}+\sqrt{t-t_1}} + 2R\frac{|t-\tau|}{\sqrt{t-t_0}+\sqrt{\tau-t_0}} \\
        &\leq 2R\frac{|\tau -t|}{\sqrt{|\tau-t|}} + 2R\frac{|t-\tau|}{\sqrt{|t-\tau|}} \leq 4R\sqrt{|\tau-t|} \\
        &<\varepsilon.
        \end{align*}
        We find a bound on $P$ as follows
        \begin{align*}
            |(P\Phi)(t)| &\leq \| y_0 + \int_{t_1}^t \eta(s)+A_u(\xi(s),s)\,ds \|_\infty \\
            &\qquad +\| w_0 + \int_{t_1}^t \left(\mu + \frac{\kappa\mu^{1/2}}{\sqrt{t-s}} + M_u(\xi(s),s)\right)\eta(s) + B_u(\xi(s),s)\,ds \|_\infty \\
            &\qquad +\|\kappa\mu^{1/2}\int_{t_0}^{t_1}\frac{w(s)}{\sqrt{t_1-s}} - \frac{w(s)}{\sqrt{t-s}}\,ds\|_\infty.
        \end{align*}
        Let us examine the integral in the final term as follows
        \begin{align*}
            \int_{t_0}^{t_1}\frac{w(s)}{\sqrt{t_1-s}} - \frac{w(s)}{\sqrt{t-s}}\,ds &= \int_{t_0}^{t_1}\frac{w(s)}{\sqrt{t-s}\sqrt{t_1-s}}(\sqrt{t-s}-\sqrt{t_1-s})\,ds \\
            &= \int_{t_0}^{t_1}\frac{w(s)(t-t_1)}{\sqrt{t-s}\sqrt{t_1-s}(\sqrt{t-s}+\sqrt{t_1-s})},
        \end{align*}
        and we can bound this using $|t-t_1|<\delta$, $\sqrt{t-s}\geq \sqrt{t-t_1}$ and $\sqrt{t-s}+\sqrt{t_1-s} \geq \sqrt{t-t_1} > \sqrt{\delta}$:
        \begin{align*}
            \left\|\int_{t_0}^{t_1}\frac{w(s)}{\sqrt{t_1-s}} - \frac{w(s)}{\sqrt{t-s}}\,ds\right\|_\infty &\leq \left\|\int_{t_0}^{t_1}\frac{w(s)\delta}{\sqrt{t-s}\sqrt{t_1-s}\sqrt{\delta}} \,ds \right\|_\infty
            \leq \left\| \int_{t_0}^{t_1}\frac{w(s)}{\sqrt{t_1-s}}\,ds \right\|_\infty \\
            &\leq R\left\| \int_{t_0}^{t_1}\frac{1}{\sqrt{t_1-s}}\,ds \right\|_\infty \leq 2R\sqrt{T-t_0}.
        \end{align*}
        Recall that, from assumption \eqref{assumptions}, we have for any $x_1,x_2\in\mathbb{R}^n$ and $\tau_1,\tau_2\in\mathbb{R}$
        \begin{align*}
            |A_u(x_1,\tau_1) - A_u(x_2,\tau_2)| &\leq |A_u(x_1,\tau_1) - A_u(x_2,\tau_1)| + |A_u(x_2,\tau_1) - A_u(x_2,\tau_2)| \\
            &\leq \|\nabla A_u\|_\infty|x_1-x_2| + \|\partial_t A_u\|_\infty|\tau_1-\tau_2| \\
            &\leq L_b\left( |x_1-x_2|+|\tau_1-\tau_2| \right)\,,\\
            |B_u(x_1,\tau_1) - B_u(x_2,\tau_2)| &\leq L_b\left( |x_1-x_2|+|\tau_1-\tau_2| \right) \,,
        \end{align*}
        by the mean value theorem.
        By integrating from $t_1$ to $t$,
        \begin{align*}
            \left\| \int_{t_1}^t A_u(\xi(s),s) \right\|_\infty &\leq \int_{t_1}^t L_b\left( |\xi(s)|+|s-t_0| \right) + |A_u(0,t_0)| \,ds \\
            &\leq L_b\|\xi\|_\infty (t-t_1) + L_b\delta(t-t_1)  + |A_u(0,t_0)|(t-t_1) \\
            \left\| \int_{t_1}^t B_u(\xi(s),s) \right\|_\infty &\leq L_b\|\xi\|_\infty (t-t_1) + L_b\delta(t-t_1)  + |B_u(0,t_0)|(t-t_1)
        \end{align*}
        Hence we may improve our bound
        \begin{align*}
            |(P\Phi)(t)| &\leq |y_{t_1}| + |w_{t_1}| + \|\Phi\|_\infty\left((t-t_1)+\mu(t-t_1)+ 2\kappa\mu^{1/2}\sqrt{t-t_1}+ L_b(t-t_1) + 2L_b(t-t_1) \right) \\
            &\qquad  + 2R\sqrt{T-t_0} + (t-t_1)\big[2L_b\delta + A_u(0,t_0) + B_u(0,t_0)\big]\,.
        \end{align*}
        Setting  $K = 4\max\{R,2R\sqrt{T-t_0}\}$ and $\delta$ such that
        \begin{equation*}\label{DeltaConditions1}
            \delta+\mu\delta + 2\kappa\mu^{1/2}\sqrt{\delta} + 3L_b\delta  < 1/4,\quad (2L_b\delta + A_u(0,t_0) + B_u(0,t_0))\delta < K/4,
        \end{equation*}
        we have that our lemma holds.
    \end{proof}
    
    \begin{lemma}\label{NewFarazmandLemma2}
        For $P$ as defined by equation \eqref{Pmap}, there exists $\delta$ such that for any $\Phi_1,\Phi_2 \in X_{\delta,K}$ we have
        \begin{equation}
            \|P\Phi_1-P\Phi_2\|_\infty \leq \frac{1}{2}\|\Phi_1-\Phi_2\|_\infty.
        \end{equation}
    \end{lemma}
    \begin{proof}
        The proof of this is as in Lemma 2 in \cite{Farazmand2015}, since in $P\Phi_1$ and $P\Phi_2$ the integral from $t_0$ to $t_1$ is the same and thus cancels. Thus the proof exactly follows that of the standard Maxey-Riley system without additional memory, with no modifications necessary since the boundedness of $A_u$ and $B_u$ is not used. 
        Thus this lemma holds for $\delta$ sufficiently small to ensure that
        \begin{equation}\label{DeltaConditions2}
            \delta + \mu\delta + 2\kappa\mu^{1/2}\sqrt{\delta} + L_b\delta < 1/4,\quad (2+K)L_c\delta < 1/4\,.
        \end{equation}
    \end{proof}

    \section{Proving the results from sections \ref{GLOBALWELLPOSEDNESSSECTION} and \ref{WEAKTOSTRONGSOLUTIONS}}\label{TechnicalProofsAppendix}
    
\begin{repeat_proposition}[\ref{UniquenessProposition}]
    Suppose $(y_1,w_1)$ and $(y_2,w_2)$ are two solutions to \eqref{MR_integrated_ODEs} with domains $[t_0,T_1)$ and $[t_0,T_2)$ respectively, corresponding to the same initial condition $(y_0,w_0)$, then the two solutions coincide on $[t_0,\min\{T_1,T_2\})$.
\end{repeat_proposition}

\begin{proof}
For $t$ in $[t_0,\min\{T_1,T_2\})$ we have
\begin{equation}
    \begin{aligned}
        y_i(t) &= y_0 + \int_{t_0}^t w_i(s) + A_u(y_i(s),s)\,ds, \quad i = 1,2 \\
        w_i(t) &= w_0 + \int_{t_0}^t \left( -\mu w_i(s) - M_u(y_i(s),s)w_i(s)-\kappa\mu^{1/2}\frac{w_i(s)}{\sqrt{t-s}} + B_u(y_i(s),s)\right)\,ds, \quad i = 1,2\,.
    \end{aligned}    
\end{equation}
We now consider the Euclidean norm of the differences $\|y_1-y_2\|$ and $\|w_1-w_2\|$, and find a bound on these as follows:
\begin{equation}
\begin{aligned}
    \|y_1(t) - y_2(t) \| &= \left\| \int_{t_0}^t w_1(s) - w_2(s) + A_u(y_1(s),s) - A_u(y_2(s),s) \,ds \right\|, \\
    &\leq \int_{t_0}^t \|w_1(s)-w_2(s)\| + L_c\|y_1(s)-y_2(s)\|\,ds
\end{aligned}
\end{equation}
and
\begin{equation}
\begin{aligned}
    \|w_1(t)-w_2(t)\| &= \bigg\| \int_{t_0}^t -\mu(w_1(s)-w_2(s)) - M_u(y_1(s),s)w_1(s) + M_u(y_2(s),s)w_2(s) \\
    &\qquad  -\kappa\mu^{1/2}\left(\frac{w_1(s)}{\sqrt{t-s}}-\frac{w_2(s)}{\sqrt{t-s}}\right) +B_u(y_1(s),s)-B_u(y_2(s),s)\,ds \bigg\| \\
    &\leq \int_{t_0}^t \mu\|w_1(s)-w_2(s)\| + L_b\|w_1(s)-w_2(s)\|  \\
    &\qquad  -\frac{\kappa\mu^{1/2}}{\sqrt{t-s}}\|w_1(s)-w_2(s)\| + L_c\|y_1(s)-y_2(s)\| \,ds.
\end{aligned}
\end{equation}
We have uniqueness from an application of the Gr\"onwall result from Appendix \ref{GronwallAppendix} with `$u(t)$' equal to $\|y_1(t)-y_2(t)\|+\|w_1(t)+w_2(t)\|$, noting that $a(t) = 0$ in the case of the above bounds.
\end{proof}

\begin{repeat_lemma}[\ref{UnionSolutionsLemma}]
    Let $\{(y_\alpha(t),w_\alpha(t))\}_{\alpha\in A}$ be a family of solutions to \eqref{MR_integrated_ODEs} with initial condition $(y_0,w_0)$, where $A$ is an arbitrary index set. Let the domain of $(y_\alpha,w_\alpha)$ be $[t_0,T_\alpha)$. We define $T$ such that $[t_0,T) = \bigcup_{\alpha\in A}[t_0,T_\alpha)$, and then define a function on $[t_0,T)$ by
    \begin{equation}\tag{\ref{LemmaSolution} revisited}
        (y(t),w(t)) =(y_\alpha(t),w_\alpha(t)),\quad \text{if }t\in [t_0,T_\alpha).
    \end{equation}
    Then $(y(t),w(t))$ is also a solution to \eqref{MR_integrated_ODEs} with the same initial condition.
\end{repeat_lemma}

\begin{proof}
    We must justify first that \eqref{LemmaSolution} gives a consistent definition of $(y,w)$, i.e. that $(y(t),w(t))$ does not depend on the choice of $\alpha$. For $t\in [t_0,T_{\alpha_1})$, \eqref{LemmaSolution} gives that $(y(t),w(t)) =(y_{\alpha_1}(t),w_{\alpha_1}(t))$. If $t$ also belongs to $[t_0,T_{\alpha_2})$, then $t\in [t_0,\min\{T_{\alpha_1},T_{\alpha_2}\})$ and therefore our uniqueness result Proposition \ref{UniquenessProposition} implies that $(y_{\alpha_1}(t),w_{\alpha_1}(t))=(y_{\alpha_2}(t),w_{\alpha_2}(t))$ for this value of $t$.
    
    Now we prove that $(y,w)$ defined by \eqref{LemmaSolution} defines a solution to \eqref{MR_integrated_ODEs} on $[t_0,T)$. We know that $t_0\in [t_0,T_\alpha)$ for any $\alpha$ and therefore
    \begin{equation}
        (y(t_0),w(t_0)) = (y_\alpha(t_0),w_\alpha(t_0)) = (y_0,w_0),
    \end{equation}
    since $(y_\alpha,w_\alpha)$ is a solution to \eqref{MR_integrated_ODEs} with initial condition $(y_0,w_0)$. Furthermore, for any $t\in[t_0,T)$ there exists $\alpha$ such that $t\in [t_0,T_\alpha)$. We know that $(y_\alpha,w_\alpha)$ solves \eqref{MR_integrated_ODEs} on $[t_0,T_\alpha)$, $(y,w)=(y_\alpha,w_\alpha)$ on $[t_0,T_\alpha)$, and therefore $(y,w)$ solves \eqref{MR_integrated_ODEs} at any $t\in [t_0,T)$.
\end{proof}

\begin{repeat_lemma}[\ref{GronwallLemma}]
    If $u$ satisfies the conditions \eqref{assumptions} and $(y,w)$ satisfies \eqref{MR_integrated_ODEs}, on the interval $[t_0,T)$, there exists some $(C_Y,C_W)$ depending on $T,y_0,w_0,\kappa,\mu$ and $L_b$ such that \begin{align}
        \sup_{t\in [t_0,T)}|y(t)|&\leq C_Y,\\
        \sup_{t\in [t_0,T)}|w(t)|&\leq C_W.
    \end{align} 
\end{repeat_lemma}
\begin{proof}
    We seek to apply a bound on the solution using the integrated form of the equation. We first notice that, as in Appendix \ref{AppendixFarazmandLemmas}, we may bound the integral of $B_u(y(s),s)$ using the Lipschitz property. That is, for any $x_1,x_2\in\mathbb{R}^{n}$ and $\tau_1,\tau_2\in\mathbb{R}$, we have
    \begin{align*}
        |B_u(x_1,\tau_1)-B_u(x_2,\tau_2)| &\leq |B_u(x_1,\tau_1) - B_u(x_2,\tau_1)| + |B_u(x_2,\tau_1) - B_u(x_2,\tau_2) |
        \\
        &\leq \|\nabla B_u \|_{\infty} |x_1-x_2| + \|\partial_tB_u\|_{\infty} |\tau_1-\tau_2|
        \\
        &\leq L_b(|x_1-x_2|+|\tau_1-\tau_2|)
        \,.
    \end{align*}
    Therefore, choosing $x_1 = y(s)$, $\tau_1 = s$, and $x_2=\tau_2 = 0$, we have that
    \begin{align*}
        \bigg|\int_{t_0}^t B_u(y(s),s)\,ds\bigg| &\leq \int_{t_0}^t L_b(|y(s)|+|s|) + |B_u(0,0)|\,ds
        \\
        &\leq L_b\int_{t_0}^t|y(s)|\,ds + L_b\int_{t_0}^t|s|\,ds + |B_u(0,0)|(t-t_0)
        \\
        &\leq L_b\int_{t_0}^t|y(s)|\,ds + \frac{L_b(t^2-t_0^2)}{2} + |B_u(0,0)|(t-t_0)
        \,.
    \end{align*}
    Beginning with the equation for $w$, we seek a bound on the solution as follows
	\begin{align*}
		|w(t)| &\leq |w_0| + \int_{t_0}^t\Big|-\mu w(s) - M_u(y(s),s)w(s) - \kappa \mu^{1/2}\frac{w(s)}{\sqrt{t-s}} + B_u(y(s),s) \Big|\,ds \\
		&\leq |w_0| + \int_{t_0}^t |\mu+L_b||w(s)| + \Big|\kappa\mu^{1/2}\frac{w(s)}{\sqrt{t-s}} \Big| + |B_u(y(s),s)|\,ds \\
		&\leq |w_0| + |\mu+L_b|\int_{t_0}^t|w(s)|\,ds + |\kappa\mu^{1/2}|\int_{t_0}^t(t-s)^{-1/2}|w(s)|\,ds \\
		&\qquad\qquad + L_b\int_{t_0}^t|y(s)|\,ds + \frac{L_b(t^2-t_0^2)}{2} + |B_u(0,0)|(t-t_0)
		\,.
	\end{align*}
	We now proceed with the equation for $y$
	\begin{align*}
	    |y(t)| &\leq |y_0| + \int_{t_0}^t |w(s)| + |A_u(y(s),s)|\,ds
	    \\
	    &\leq \int_{t_0}^t |w(s)|\,ds + L_b\int_{t_0}^t|y(s)|\,ds + \frac{L_b(t^2-t_0^2)}{2} + |A_u(0,0)|(t-t_0)
	    \,.
	\end{align*}
	We then must consider the above inequalities for $|y(s)|$ and $|w(s)|$ as a pair. In particular, we define $a(s)$ by
	\begin{equation*}
	    a(s) = |y(s)|+|w(s)|\,,
	\end{equation*}
	and then we have the following inequality
	\begin{equation}
	\begin{aligned}
	    a(s) &\leq a(t_0) + L_b(t^2-t_0^2) + (|A_u(0,0)| + |B_u(0,0)|)(t-t_0)
	    \\
	    &\qquad + \max\{1,\mu+L_b\}\int_{t_0}^t a(s)\,ds + |\kappa\mu^{1/2}|\int_{t_0}^t(t-s)^{-1/2}a(s)\,ds\,.
	\end{aligned}
	\end{equation}
    A Gr\"onwall-style result of S. Y. Lin \cite{Lin2013a}, gives a bound on the solution as required (see Appendix \ref{GronwallAppendix}).

\end{proof}

\begin{repeat_lemma}[\ref{HolderLimitLemma}]
    Assuming $(y,w)$ is a solution of \eqref{MR_integrated_ODEs}, we have that
    \begin{itemize}
    \item[(i)] the integral \eqref{MRFractionalIntegral} is 1/2-H\"older, i.e. there exists some constant $C>0$ such that for any $t_1,t_2$ with $t_0<t_1<t_2$ we have
\begin{equation}\tag{\ref{HolderStatementIntegral} revisited}
	\left| \int_{t_0}^{t_2}\frac{w(s)}{\sqrt{t_2-s}}\,ds -\int_{t_0}^{t_1}\frac{w(s)}{\sqrt{t_1-s}}\,ds \right| \leq C|t_2-t_1|^{1/2}\,,
\end{equation}

and
    \item[(ii)] if $w(t_0) = 0$ then the following limit exists and is equal to zero
    \begin{equation}
        \lim_{t\to t_0}\frac{w(t)}{\sqrt{t-t_0}} = 0\,.
    \end{equation}
    \end{itemize}
\end{repeat_lemma}
\begin{proof}
We prove the two parts separately, beginning with part (i).

\emph{Part (i):} 
We bound the left hand side of \eqref{HolderStatementIntegral} as follows
\begin{equation*}
	\left| \int_{t_0}^{t_2}\frac{w(s)}{\sqrt{t_2-s}}\,ds -\int_{t_0}^{t_1}\frac{w(s)}{\sqrt{t_1-s}}\,ds \right| \leq \left| \int_{t_1}^{t_2}\frac{w(s)}{\sqrt{t_2-s}}\,ds\right| + \left|\int_{t_0}^{t_1}\frac{w(s)}{\sqrt{t_2-s}}-\frac{w(s)}{\sqrt{t_1-s}}\,ds \right|,
\end{equation*}
and it remains to prove that both terms on the right hand side are indeed 1/2-H\"older continuous. By the same argument as \eqref{IntegralContinuityt_0}, the first term is 1/2-H\"older and so is the second term since
\begin{align*}
	\left|\int_{t_0}^{t_1}\frac{w(s)}{\sqrt{t_2-s}}-\frac{w(s)}{\sqrt{t_1-s}}\,ds \right| &\leq K\left| \int_{t_0}^{t_1}\frac{1}{\sqrt{t_2-s}}-\frac{1}{\sqrt{t_1-s}}\,ds \right| \\
	&\leq \left| -2K\sqrt{t_2-t_1} + 2K\sqrt{t_2-t_0} - 2K\sqrt{t_1-t_0} \right| \\
	&\leq 2K\sqrt{t_2-t_1} + 2K\left| \sqrt{t_2-t_0} - \sqrt{t_1-t_0} \right| \\
	&\leq 2K\sqrt{t_2-t_1} + 2K\frac{|t_2-t_1|}{\sqrt{t_2-t_0} + \sqrt{t_1-t_0}} \\
	&\leq 2K\sqrt{t_2-t_1} + 2K\frac{|t_2-t_1|}{\sqrt{t_2-t_1} + \sqrt{t_1-t_1}} \\
	&\leq 4K\sqrt{t_2-t_1},
\end{align*}
where in the first line we have made use of the boundedness property of solutions to the Maxey-Riley equation, see Remark \ref{FarazmandTheorem}.
    
\emph{Part (ii):}
    Recall that $w$ satisfies
   \begin{equation}
        w(t) = w(t_0) + \int_{t_0}^t \left( -\mu w(s) - M_u(y(s),s)w(s)-\kappa\mu^{1/2}\frac{w(s)}{\sqrt{t-s}} + B_u(y(s),s)\right)\,ds\,.
   \end{equation}
   Dividing through by $\sqrt{t-t_0}$ and considering $w(t_0) = 0$, we have
   \begin{equation}
        \frac{w(t)}{\sqrt{t-t_0}} = \frac{1}{\sqrt{t-t_0}}\int_{t_0}^t  \left( -\mu w(s) - M_u(y(s),s)w(s) + B_u(y(s),s)\right)\,ds - \frac{\kappa\mu^{1/2}}{\sqrt{t-t_0}}\int_{t_0}^t\frac{w(s)}{\sqrt{t-s}}\,ds\,.
   \end{equation}
   We know that $w$ is locally bounded and, under assumptions \eqref{assumptions}, $M_u$ is uniformly bounded and $B_u$ is sufficiently smooth to ensure that it is locally bounded on each interval $[t_0,t]$. Hence, there exists some $c_1$ which (locally) bounds the integrand of the first integral on the right hand side and hence
   \begin{equation}
       \frac{1}{\sqrt{t-t_0}}\int_{t_0}^t  \left( -\mu w(s) - M_u(y(s),s)w(s) + B_u(y(s),s)\right)\,ds \leq c_1\sqrt{t-t_0}\,.
   \end{equation}
   We may therefore deduce that
   \begin{equation}
       \lim_{t\to t_0}\frac{1}{\sqrt{t-t_0}}\int_{t_0}^t  \left( -\mu w(s) - M_u(y(s),s)w(s) + B_u(y(s),s)\right)\,ds = 0\,.
   \end{equation}
   It remains to show the existence of the limit 
  \begin{equation}\label{LemmaSubLimit}
    \begin{aligned}
      \lim_{t\to t_0}& \frac{1}{\sqrt{t-t_0}}\int_{t_0}^t \frac{w(s)}{\sqrt{t-s}}\,ds = \\
      &\lim_{t\to t_0}\bigg[\frac{1}{\sqrt{t-t_0}}\int_{t_0}^t\frac{1}{\sqrt{t-s}}\int_{t_0}^s\left(-\mu w(r) - M_u(y(r),r)w(r) + B_u(y(r),r)\right)\,dr\,ds\bigg] \\
      &-\lim_{t\to t_0}\bigg[\frac{1}{\sqrt{t-t_0}}\int_{t_0}^t\frac{\kappa\mu^{1/2}}{\sqrt{t-s}}\int_{t_0}^s\frac{w(r)}{\sqrt{s-r}}\,dr\,ds
      \bigg]\,.
    \end{aligned}
  \end{equation}
  where to show the equality we have used the equation \eqref{MR_integrated_ODEs} for $w(s)$. We may bound the middle line of \eqref{LemmaSubLimit} by
  \begin{equation*}
      \frac{1}{\sqrt{t-t_0}}\int_{t_0}^t\frac{1}{\sqrt{t-s}}\int_{t_0}^s\left(-\mu w(r) - M_u(y(r),r)w(r) + B_u(y(r),r)\right)\,dr\,ds \leq \frac{c_1}{\sqrt{t-t_0}}\int_{t_0}^t\frac{s-t_0}{\sqrt{t-s}}\,ds\,.
  \end{equation*}
  The integral on the right hand side may be calculated by making a suitable substitution
  \begin{align*}
      \int_{t_0}^t \frac{s-t_0}{\sqrt{t-s}}\,ds &= -\int_{\sqrt{t-t_0}}^0 \frac{t-t_0 - u^2}{u}2u\,du = \int_0^{\sqrt{t-t_0}} (t-t_0) - u^2\,du \\
      &= 2\left( (t-t_0)^{3/2} - \frac{1}{3} (t-t_0)^{3/2} \right) = \frac43 (t-t_0)^{3/2}\,.
  \end{align*}
  Therefore, the limit in the second line of \eqref{LemmaSubLimit} is equal to 0. We now observe the limit in the final line of \eqref{LemmaSubLimit}, i.e.
  \begin{equation}
      \lim_{t\to t_0}\bigg[ \frac{1}{\sqrt{t-t_0}}\int_{t_0}^t\frac{1}{\sqrt{t-s}}\int_{t_0}^s\frac{w(r)}{\sqrt{s-r}}\,dr\,ds\bigg]\,.
  \end{equation}
  First recall that $w$ is bounded by $K$, then we may bound and again calculate by making a series of suitable substitutions
  \begin{equation}
      \begin{aligned}
          \frac{1}{\sqrt{t-t_0}}\int_{t_0}^t\frac{1}{\sqrt{t-s}}\int_{t_0}^s\frac{w(r)}{\sqrt{s-r}}\,dr\,ds &\leq \frac{K}{\sqrt{t-t_0}}\int_{t_0}^t\frac{1}{\sqrt{t-s}}\int_{t_0}^s\frac{1}{\sqrt{s-r}}\,dr\,ds \\
          &\leq \frac{2K}{\sqrt{t-t_0}}\int_{t_0}^t\frac{\sqrt{s-t_0}}{\sqrt{t-s}}\,ds \\
          &\leq \frac{2K}{\sqrt{t-t_0}}\int_0^{\sqrt{t-t_0}} \frac{2u^2}{\sqrt{t-u^2-t_0}}\,du \\
          &\leq 4K\sqrt{t-t_0}\int_0^{\pi/2}\sin^2\theta\,d\theta = K\pi \sqrt{t-t_0}\,,
      \end{aligned}
  \end{equation}
  and thus
  \begin{equation}
      \lim_{t\to t_0}\left[ \frac{1}{\sqrt{t-t_0}}\int_{t_0}^t\frac{1}{\sqrt{t-s}}\int_{t_0}^s\frac{w(r)}{\sqrt{s-r}}\,dr\,ds \right] = 0\,.
  \end{equation}
  
  Putting these calculations together, we have proven our claim.
\end{proof}

    \section{Fractional evolution of a matrix inverse}\label{MatrixInverseAppendix}
    
    If an $n\times n$ matrix, $M_t$, evolves according to
    \begin{equation*}
        M_t  = I_n + \int_{t_0}^t A_sM_s\,ds\,,
    \end{equation*}
    and another matrix, $N_t$, according to
    \begin{equation*}
        N_t = I_n - \int_{t_0}^t N_sA_s\,ds\,.
    \end{equation*}
    Then we have
    \begin{align*}
        N_tM_t &= I_n + \int_{t_0}^t N_s\frac{dM_s}{ds}\,ds + \int_{t_0}^t \frac{dN_s}{ds}M_s\,ds \\
        &= I_n + \int_{t_0}^t N_sA_sM_s\,ds - \int_{t_0}^t N_sA_sM_s\,ds = I_n\,.
    \end{align*}
    Thus, $\det(N_tM_t) = \det(N_t)\det(M_t) = 1$ and therefore $\det(M_t) \neq 0$ and $M_t$ is invertible. Moreover, $N_t$ is the inverse of $M_t$ for all $t$. 
    
    In the fractional case, this is less simple. Suppose instead that $M_t$ evolves according to
    \begin{equation}
        M_t = I_n + \frac{1}{\sqrt{\pi}} \int_{t_0}^t \frac{A_sM_s}{\sqrt{t-s}}\,ds\,,
    \end{equation}
    so, in differential form,
    \begin{equation}
        \frac{dM_t}{dt} = _{t_0}D^{1/2}(A_tM_t)\,.
    \end{equation}
    Recall that the \emph{right} Riemann-Liouville fractional derivative may be defined similarly to the standard, or \emph{left} Riemann-Liouville derivative in Definition \ref{RiemannLiouvilleDefinition}, by instead changing the equation \ref{RiemannLiouvilleDefinitionEquation} to
    \begin{equation}
		D^p_bf(t) = \frac{1}{\Gamma(n-p)}\left(-\frac{d}{dt}\right)^n\int_t^b(s-t)^{n-p-1}f(s)\,ds.
	\end{equation}
	We then note that, for continuous functions $f$ and $g$, we have the following result
	\begin{equation}\label{FractionalProductRule}
	    \int_a^b f(t) _aD^\alpha g(t) \,ds = \int_a^b g(s)D^\alpha_b f(s)\,ds\,.
	\end{equation}
	Suppose now that the matrix $N_s$ evolves by
	\begin{equation}
	    N_s = I_n + \frac{1}{\sqrt{\pi}} \int_s^t \frac{N_rA_r}{\sqrt{r-s}}\,dr\,,
	\end{equation}
	where this evolution depends on time. In differential form this is
	\begin{equation}
	    \frac{dN_s}{ds} = - D_t^{1/2} (N_sA_s)\,.
	\end{equation}
	Now suppose that $A_t = A$ is constant, then we have
	\begin{equation}
	\begin{aligned}
	    N_t M_t &= I_n + \int_{t_0}^t N_s\frac{dM_s}{ds}\,ds + \int_{t_0}^t \frac{dN_s}{ds}M_s\,ds\,, \\
	    &= I_n + \int_{t_0}^t N_s {_{t_0}D^{1/2}} (AM_s)\,ds - \int_{t_0}^t D^{1/2}_t(N_sA)M_s\,ds\,.
	   \end{aligned}
	\end{equation}
	Now by the equation \ref{FractionalProductRule}, we have
	\begin{equation}
	    N_tM_t = I_n\,,
	\end{equation}
	and thus we have determined the equation for the inverse when $M_t$ is a matrix evolving according to a fractional differential equation.
\end{appendices}

\bibliography{PhDProject.bib}
\bibliographystyle{abbrv}

\end{document}